\documentclass[12pt]{article}
    \usepackage{amsmath}
    \usepackage{amssymb}
    \usepackage{amsthm}

\usepackage{pstricks}

\usepackage{graphicx,psfrag}

      \newcommand {\al}   {\alpha}          \newcommand {\bt}  {\beta}

              \newcommand {\ve}  {\varepsilon}

      \newcommand {\pl}   {\partial}        
      
           \newcommand {\UUU}  {{\cal U}}
                       
      \newcommand {\RRR}  {{\mathbb R}}     \newcommand {\SSS}  {S}

      \newcommand {\ZZZ}  {{\mathbb Z}}

      \newcommand {\bbb}  {b}      \newcommand {\iii}  {k}
       \newcommand {\psystem}  {p-system}

     \newcommand {\beq}  {\begin{equation}}
      \newcommand {\eeq}  {\end{equation}}
     \newcommand {\beqo}  {\begin{equation*}}
      \newcommand {\eeqo}  {\end{equation*}}

      \newtheorem{theorem}{Theorem}
      \newtheorem{lemma}{Lemma}

      \newtheorem*{opr}{Definition}

\title{Plane sets invisible in finitely many directions}

\author{Alexander Plakhov\thanks{Center for R\&{}D in Mathematics and Applications, Department of Mathematics, University of Aveiro, Portugal and Institute for Information Transmission Problems, Moscow, Russia.}}

\begin{document}

\maketitle

\begin{abstract}
We consider the problem of mirror invisibility for plane sets. Given a circle and a finite number of unit vectors (defining the directions of invisibility) such that the angles between them are commensurable with $\pi$, for any $\ve > 0$ there exists a set invisible in the chosen directions that contains the circle and is contained in its $\ve$-neighborhood. This set is the disjoint union of infinitely many domains with piecewise smooth boundary.
\end{abstract}

\begin{quote}
\end{quote}

\begin{quote}
\end{quote}

\section{Introduction}

Nowadays, the problem of invisibility and its practical realization attracts a lot of attention of mathematical, physical, and technological communities. Some approaches to constructing invisibility devices are based on rapidly growing metamaterial technologies \cite{chen,gh,pendry,schurig,zhang}. A very interesting approach, in the framework of geometric optics, is related to creating a transparent refracting (with variable refraction index) coating for objects to be hidden \cite{leon1,leon2}.

Here we deal with billiard invisibility, when a mirrored object (a body) is represented by a bounded domain or by a disjoint union of (finitely or countably many) domains in Euclidean space, and propagation of light rays is represented by motion of billiard particles outside the body.

In what follows we consider billiard trajectories $x(t)$ parameterized by time $t$ and representing the motion, uniform between consecutive reflections, of a billiard particle. We concentrate on trajectories that are defined for all $t \in \RRR$ and have a finite number of reflections.

Take a bounded plane set $B \subset \RRR^2$ and consider the billiard in its complement $B^c = \RRR^2 \setminus B$. Slightly abusing the language, we say that a parameterized billiard trajectory $x(t), \, t \in \RRR$ is {\it invisible}, if the initial and final (semi-infinite) segments of the trajectory lie on the same straight line and are passed in the same direction. In other words, let $x(t) = b + vt$ for $t$ sufficiently small; the trajectory is invisible, if for $t$ sufficiently large and for a certain constant $\tau \in \RRR$ we have $x(t) = b + v(t-\tau)$. The corresponding straight line is called to be {\it associated} with the invisible trajectory (see Fig.~\ref{fig:invisTraj}).

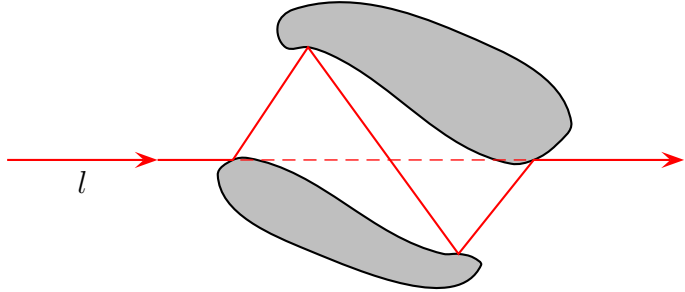
\begin{figure}
\begin{picture}(0,110)
\scalebox{1}{
\rput(7,1.4){

\psecurve[fillstyle=solid,fillcolor=lightgray]
(3.5,0.5)(2,1.7)(-0.2,2)(-0.4,1.8)(-0.3,1.5)(0,1.5)(0.3,1.4)(2.7,-0.05)(3,0)(3.3,0.2)(3.5,0.5)(2,1.7)(-0.2,2)
\psecurve[fillstyle=solid,fillcolor=lightgray]
(2.3,-1.4)(0,-1.2)(-1.2,-0.2)(-1,0)(-0.8,0.025)(1.8,-1.25)(2,-1.25)(2.2,-1.3)(2.3,-1.4)(0,-1.2)(-1.2,-0.2)
\psline[linecolor=red,arrows=->,arrowscale=2,linewidth=0.8pt](-4,0)(-2,0)
\psline[linecolor=red,arrows=->,arrowscale=2,linewidth=0.8pt](-2,0)(-1,0)(0,1.5)(2,-1.25)(3,0)(5,0)
\psline[linecolor=red,linestyle=dashed,linewidth=0.5pt](-0.9,0)(2.9,0)
\rput(-3,-0.3){\scalebox{1}{$l$}}
}
}
\end{picture}
\caption{An invisible trajectory and the associated straight line $l$.}
\label{fig:invisTraj}
\end{figure}

Let $v \in S^1$. We say that $B$ is {\it invisible in the direction of} $v$ (or {\it along the vector} $v$), if almost all\footnote{with respect to the natural one-dimensional Lebesgue measure} straight lines with the director vector $v$ are associated with invisible trajectories. Obviously, $B$ is invisible along $v$ if and only if it is invisible along $-v$. An example of a set invisible in a direction is provided in Fig.~\ref{fig:invisDir}; for more details see \cite{0-resist}.

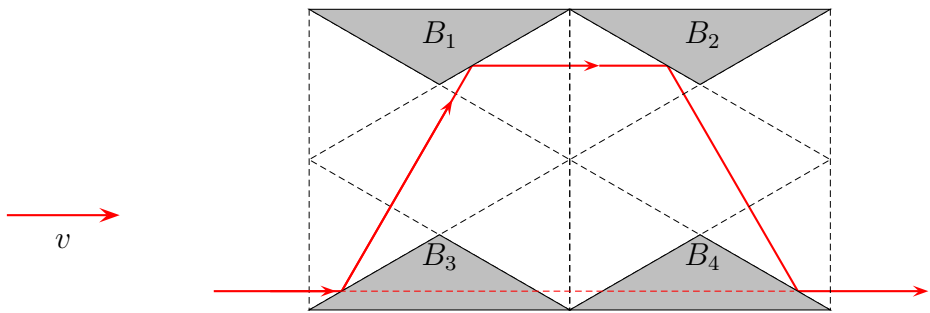
\begin{figure}[h]
\begin{picture}(0,140)

\rput{90}(8,-1.75){
\scalebox{0.5}{
\rput(8.7,4){
\pspolygon[linestyle=dashed,linewidth=0.3pt](-4,-3.4641)(4,-3.4641)(0,3.4641)
\pspolygon[linestyle=dashed,linewidth=0.3pt](-4,3.4641)(4,3.4641)(0,-3.4641)
 \psline[linewidth=1.6pt,linecolor=red,arrows=->,arrowscale=1.8](-3.5,6)(-3.5,2.75)
\psline[linewidth=1.6pt,linecolor=red,arrows=->,arrowscale=1.8](-3.5,4.5)(-3.5,2.598075)(1.6,-0.34641)
\psline[linewidth=1.6pt,linecolor=red,arrows=->,arrowscale=1.8](-3.5,2.598075)(2.5,-0.866025)(2.5,-4.25)
\psline[linewidth=1.6pt,linecolor=red,arrows=->,arrowscale=1.8](2.5,-4.25)(2.5,-6.06)(-3.5,-9.53)(-3.5,-13)

\pspolygon[fillstyle=solid,fillcolor=lightgray](-4,-3.4641)(-4,3.4641)(-2,0)
\pspolygon[fillstyle=solid,fillcolor=lightgray](4,-3.4641)(4,3.4641)(2,0)
}
\rput(8.7,-2.93){
\pspolygon[linestyle=dashed,linewidth=0.3pt](-4,-3.4641)(4,-3.4641)(0,3.4641)
\pspolygon[linestyle=dashed,linewidth=0.3pt](-4,3.4641)(4,3.4641)(0,-3.4641)

\pspolygon[fillstyle=solid,fillcolor=lightgray](-4,-3.4641)(-4,3.4641)(-2,0)
\pspolygon[fillstyle=solid,fillcolor=lightgray](4,-3.4641)(4,3.4641)(2,0)

\psline[linecolor=red,linestyle=dashed](-3.5,9.6)(-3.5,-2.5)
}
}}
 \psline[linecolor=red,arrows=->,linewidth=0.8pt,arrowscale=1.8](0.25,2)(1.75,2)
 \rput(1,1.65){\scalebox{1}{$v$}}
\rput(6,4.4){$B_1$}\rput(9.5,4.4){$B_2$}
\rput(6,1.45){$B_3$}\rput(9.5,1.45){$B_4$}

\end{picture}
\caption{The union of shaded triangles $B = B_1 \cup B_2 \cup B_3 \cup B_4$ is invisible in the direction of $v$.}
\label{fig:invisDir}
\end{figure}

Likewise, we say that $B$ is {\it invisible from a point} $O \in \RRR^2$, if almost all straight lines through $O$ are associated with invisible trajectories and, additionally, $O$ lies outside Conv$(B)$. An example of a set invisible in a direction is given in Fig.~\ref{fig:invisPoint}; for more details of this construction see \cite{invis1point} or \cite{ebook}.

\begin{figure}[h]
\begin{picture}(0,200)

\rput(6.3,3.3){
\scalebox{2}{
\rput(-0.57,0){
\scalebox{0.534}{
\psdots[dotsize=3pt](-1,0)
\rput(-1.07,-0.3){\scalebox{0.93}{$O$}}
\pspolygon[linewidth=0.03pt,linecolor=white,fillstyle=solid,fillcolor=lightgray]
(0.534,0.925)(1.871,1.732)(1.64,1.483)(1.48,1.304)(1.34,1.14)(1.223,1)(1.18,0.947)
(1.136,0.89)(1.093,0.8367)(1.0456,0.7746)(1,0.709)(0.9,0.773)(0.8,0.826)(0.7,0.87)(0.6,0.906)
\pspolygon[linewidth=0.06pt,linecolor=white,fillstyle=solid,fillcolor=lightgray]
(0.534,-0.925)(1.871,-1.732)(1.64,-1.483)(1.48,-1.304)(1.34,-1.14)(1.223,-1)(1.18,-0.947)
(1.136,-0.89)(1.093,-0.8367)(1.0456,-0.7746)(1,-0.709)(0.9,-0.773)(0.8,-0.826)(0.7,-0.87)(0.6,-0.906)
\psline[linewidth=0.2pt,linestyle=dotted,linecolor=red,dotsep=1.6pt](-1,0)(0.534,-0.925)
\psline[linewidth=0.2pt,linestyle=dotted,linecolor=red,dotsep=1.6pt](-1,0)(0.534,0.925)
\psecurve[linewidth=0.12pt](0.45,0.948)(0.534,0.925)(0.6,0.9055)(0.7,0.8689)(0.8,0.8246)(0.9,0.7714)(1,0.7071)(1.1,0.6285)
\psecurve[linewidth=0.12pt](0.45,-0.948)(0.534,-0.925)(0.6,-0.9055)(0.7,-0.8689)
(0.8,-0.8246)(0.9,-0.7714)(1,-0.7071)(1.1,-0.6285)
\pscurve[linewidth=0.12pt]
(1.871,-1.732)(1.643,-1.483)(1.483,-1.304)(1.342,-1.14)
(1.2247,-1)(1.1832,-0.9487)(1.1402,-0.8944)(1.0954,-0.8367)(1.0488,-0.7746)(1,-0.7071)
\pscurve[linewidth=0.12pt]
(1,0.7071)(1.048,0.7746)(1.0954,0.8367)(1.1402,0.8944)(1.1832,0.9487)(1.2247,1)
(1.342,1.14)(1.483,1.304)(1.643,1.483)(1.871,1.732)
\psline[linewidth=0.12pt](0.534,0.925)(1.871,1.732)
\psline[linewidth=0.12pt](0.534,-0.925)(1.871,-1.732)
}}
\psdots[dotsize=0.25pt](-1,0)
\pspolygon[linewidth=0.03pt,linecolor=white,fillstyle=solid,fillcolor=lightgray]
(0.534,0.925)(1.871,1.732)(1.64,1.483)(1.48,1.304)(1.34,1.14)(1.223,1)(1.18,0.947)
(1.136,0.89)(1.093,0.8367)(1.0456,0.7746)(1,0.709)(0.9,0.773)(0.8,0.826)(0.7,0.87)(0.6,0.906)
\pspolygon[linewidth=0.03pt,linecolor=white,fillstyle=solid,fillcolor=lightgray]
(0.534,-0.925)(1.871,-1.732)(1.64,-1.483)(1.48,-1.304)(1.34,-1.14)(1.223,-1)(1.18,-0.947)
(1.136,-0.89)(1.093,-0.8367)(1.0456,-0.7746)(1,-0.709)(0.9,-0.773)(0.8,-0.826)(0.7,-0.87)(0.6,-0.906)
\psellipse[linewidth=0.01pt,linestyle=dashed,dash=1.2pt 0.8pt](0,0)(1.414,1)
\pscurve[linewidth=0.01pt,linestyle=dashed,dash=1.2pt 0.8pt]
(1.871,-1.732)(1.643,-1.483)(1.483,-1.304)(1.342,-1.14)
(1.2247,-1)(1.1832,-0.9487)(1.1402,-0.8944)(1.0954,-0.8367)(1.0488,-0.7746)(1,-0.7071)(0.9487,-0.6325)
(0.8944,-0.5477)(0.8367,-0.4472)(0.7746,-0.3162)(0.7071,0)(0.7746,0.3162)(0.8367,0.4472)(0.8944,0.5477)
(0.9487,0.6325)(1,0.7071)(1.048,0.7746)(1.0954,0.8367)(1.1402,0.8944)(1.1832,0.9487)(1.2247,1)
(1.342,1.14)(1.483,1.304)(1.643,1.483)(1.871,1.732)
\psline[linewidth=0.15pt,linestyle=dotted,dotsep=0.8pt,linecolor=blue](-1,0)(1.8284,1)
\psline[linewidth=0.15pt,linestyle=dotted,dotsep=0.8pt,linecolor=blue](-1,0)(1.8284,-1)
\pscurve[linewidth=0.06pt]
(1.871,-1.732)(1.643,-1.483)(1.483,-1.304)(1.342,-1.14)
(1.2247,-1)(1.1832,-0.9487)(1.1402,-0.8944)(1.0954,-0.8367)(1.0488,-0.7746)(1,-0.7071)
\pscurve[linewidth=0.06pt]
(1,0.7071)(1.048,0.7746)(1.0954,0.8367)(1.1402,0.8944)(1.1832,0.9487)(1.2247,1)
(1.342,1.14)(1.483,1.304)(1.643,1.483)(1.871,1.732)
\psecurve[linewidth=0.06pt](0.45,0.948)(0.534,0.925)(0.6,0.9055)(0.7,0.8689)(0.8,0.8246)(0.9,0.7714)(1,0.7071)(1.1,0.6285)
\psecurve[linewidth=0.06pt](0.45,-0.948)(0.534,-0.925)(0.6,-0.9055)(0.7,-0.8689)
(0.8,-0.8246)(0.9,-0.7714)(1,-0.7071)(1.1,-0.6285)
\psline[linewidth=0.06pt](0.534,0.925)(1.871,1.732)
\psline[linewidth=0.06pt](0.534,-0.925)(1.871,-1.732)
\psline[linewidth=0.4pt,linecolor=red,arrows=->,arrowscale=1](-1.4765,-0.2235)(-0.5235,0.2235)
\psline[linewidth=0.3pt,linecolor=red,linestyle=dashed](-0.5235,0.2235)(1.283,1.07)
\psline[linewidth=0.4pt,linecolor=red,arrows=->,arrowscale=1]
(-0.5235,0.2235)(-0.047,0.447)(0.22,-0.571)(0.4985,-0.702)
\psline[linewidth=0.4pt,linecolor=red,arrows=->,arrowscale=1]
(0.4985,-0.702)(0.777,-0.833)(1.283,1.07)(2.2,1.5)
\psdots[dotsize=0.9pt](-1,0)
}}
\end{picture}
\caption{The union of four shaded curvilinear triangles is invisible from the point $O$.}
\label{fig:invisPoint}
\end{figure}
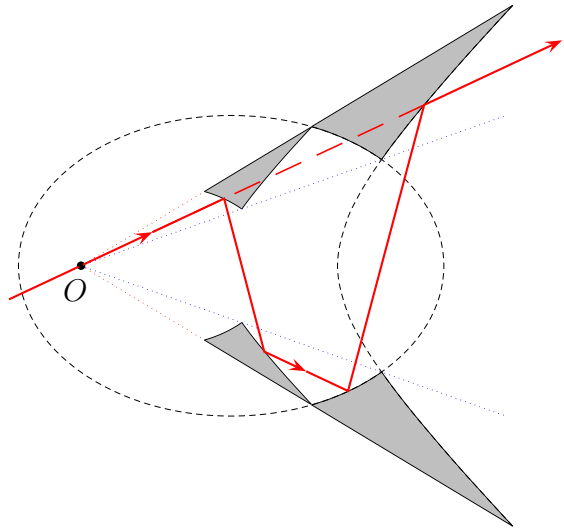

These definitions easily generalize to higher dimensions.

Invisibility in a direction can be treated as invisibility from an infinitely distant point.

In general, the problem of billiard invisibility remains open. We conjecture that

(a) given a finite set of directions and a finite set of points, there is a body (a finite or countable union of bounded domains) invisible in these directions and from these points;

(b) the set of invisible trajectories for each body has measure zero.

Conjecture (b) is closely connected with the famous Ivrii conjecture stating that the set of periodic billiard trajectories in a bounded domain has measure zero \cite{Ivrii}.

An interesting (and probably even more difficult than conjectures (a) and (b)) problem concerns the maximum dimension of the set of invisible trajectories.

Up to the moment the following is known.

$\bullet$ For any two unit vectors in $\RRR^2$ there exists a plane body invisible along these vectors \cite{invis3dir}.

$\bullet$ For any three mutually orthogonal unit vectors in $\RRR^3$ there exists a 3D body invisible along these vectors \cite{invis3dir}.

$\bullet$ For any two points in $\RRR^n \ (n \ge 2)$ there exists an $n$-dimensional body invisible from these points \cite{invis2points}.

$\bullet$ No body in $\RRR^n \ (n \ge 2)$ is invisible in all directions or, equivalently, from all points \cite{invisibility}.

The aim of this paper is to prove the following theorem.

\begin{theorem}\label{t1}
Consider a finite set of unit vectors in $\RRR^2$ such that the angle between any two vectors is commensurable with $\pi$, and let $D_1 \subset D_2 \in \RRR^2$ be two different concentric circles. Then there is a set $B$ invisible along these vectors and such that $D_1 \subset B \subset D_2$. (Here $B$ is a countably connected set, with each connected component being a domain with piecewise smooth boundary.)
\end{theorem}

We believe that the technical condition on the angles could be eventually omitted and that $D_1$ and $D_2$ could be taken to be arbitrary bounded convex bodies (with $D_1 \subset D_2$ and $\pl D_1 \cap \pl D_2 = \emptyset$). We also believe that the statement of the theorem generalizes to higher dimensions.

The rest of the paper is devoted to the proof of this theorem.

\section{Reformulation of Theorem \ref{t1} and the main idea of the proof}


Take a natural $n$ so that $\pi/n$ is a common multiplier of all angles between the directions of invisibility. We can take $n$ so large that a regular $2n$-gon $P_1$ circumscribed about $D_1$ is contained in $D_2$. Let us choose $P_1$ in such a way that a direction of invisibility is the external bisector of an angle of $P_1$ (and therefore all directions of invisibility are external bisectors of angles of $P_1$).

Extending if necessary the set of directions of invisibility, one can assume that it is composed of the external bisectors of {\it all} vertices of $P_1$. That is, we have the set of $2n$ unit vectors with the angle $\pi/n$ between nearest vectors (see Fig.~\ref{fig:concentric}). The angles between a side of $P_1$ and these vectors take the values $\pi/(2n)$,\, $3\pi/(2n),\ldots,(2n - 1)\pi/(2n)$.

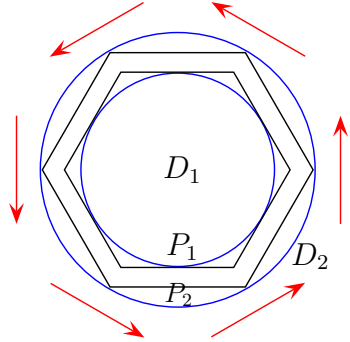
\begin{figure}[h]
\begin{picture}(0,130)
\rput(7,2.25){
\scalebox{1.8}{
\pspolygon[linewidth=0.3pt](1,0)(0.5,0.866)(-0.5,0.866)(-1,0)(-0.5,-0.866)(0.5,-0.866)
\pscircle[linewidth=0.3pt,linecolor=blue](0,0){0.72}
\pscircle[linewidth=0.3pt,linecolor=blue](0,0){1.02}
}
\scalebox{1.5}{
\rput(-0.3,0){
\pspolygon[linewidth=0.35pt](1,0)(0.5,0.866)(-0.5,0.866)(-1,0)(-0.5,-0.866)(0.5,-0.866)
}}
\scalebox{1.8}{
\rput(-0.47,0){
\psline[linewidth=0.3pt,linecolor=red,arrows=->,arrowscale=1.5](1.2,-0.4)(1.2,0.4)
}}
\rput{60}(-0.97,-0.17){
\scalebox{1.8}{
\psline[linewidth=0.3pt,linecolor=red,arrows=->,arrowscale=1.5](1.2,-0.4)(1.2,0.4)
}}
\rput{120}(-0.77,-0.17){
\scalebox{1.8}{
\psline[linewidth=0.3pt,linecolor=red,arrows=->,arrowscale=1.5](1.2,-0.4)(1.2,0.4)
}}
\rput{180}(-0.66,0){
\scalebox{1.8}{
\psline[linewidth=0.3pt,linecolor=red,arrows=->,arrowscale=1.5](1.2,-0.4)(1.2,0.4)
}}
\rput{-60}(-0.97,0.18){
\scalebox{1.8}{
\psline[linewidth=0.3pt,linecolor=red,arrows=->,arrowscale=1.5](1.2,-0.4)(1.2,0.4)
}}
\rput{-120}(-0.75,0.17){
\scalebox{1.8}{
\psline[linewidth=0.3pt,linecolor=red,arrows=->,arrowscale=1.5](1.2,-0.4)(1.2,0.4)
}}

\rput(-0.8,-1.025){\scalebox{1}{$P_1$}}
\rput(-0.85,-1.66){\scalebox{0.88}{$P_2$}}

\rput(-0.8,0){\scalebox{1}{$D_1$}}
\rput(0.9,-1.15){\scalebox{1}{$D_2$}}
}
\end{picture}
\caption{The regular hexagons $P_1$ and $P_2$ circumscribed about and inscribed in the concentric circles, $D_1$ and $D_2$, and 6 directions of invisibility.}
\label{fig:concentric}
\end{figure}

Consider the $2n$-gon $P_2$ similar to $P_1$ and inscribed in $D_2$. Obviously, $P_2$ contains $P_1$. Now Theorem \ref{t1} can be stated in the following equivalent form.

\begin{theorem}\label{t2}
Consider two similar regular $2n$-gons $P_1 \subset P_2$ and $2n$ unit vectors parallel to the external bisectors of the angles of $P_1$. Then there exists a set $B$ invisible along these vectors and such that $P_1 \subset B \subset P_2$.
\end{theorem}

Let us first explain the main idea of the proof. The key element of the underlying construction involves a collection of $n+1$ curves, a big one and $n$ small ones. We have $n$ parallel flows of particles incident on the big curve. Each incident flow (say, the $i$th one) is reflected first from the big curve and then from the corresponding (the $i$th) small one, and turns into another parallel flow, with the width much smaller than that of the original one.  That is, the $n$ incident parallel flows after two reflections are transformed into $n$ compressed parallel flows, and the compression ratio can be chosen to be arbitrarily small (see Fig.~\ref{fig:systemofcurves}).

\begin{figure}
\begin{picture}(0,100)
\scalebox{1}{
\rput(7,3.3){

\psarc[linewidth=1.2pt](0,1.06){4}{-113}{-67}
\psline[linecolor=red,arrows=->,arrowscale=2,linewidth=0.4pt](-1.4,0)(-0.95,-0.88)
\psline[linecolor=red,arrows=->,arrowscale=2,linewidth=0.4pt](-0.95,-0.88)(0.1,-2.93)(1.35,-0.35)(2.35,-0.35)
\psline[linecolor=red,arrows=->,arrowscale=2,linewidth=0.4pt](-2.9,0)(-2.433,-0.867)
\psline[linecolor=red,arrows=->,arrowscale=2,linewidth=0.4pt](-2.433,-0.867)(-1.5,-2.6)(1.8,-0.15)(2.8,-0.15)
\psline[linecolor=red,arrows=->,arrowscale=2,linewidth=0.4pt](0,0)(0.5,-0.867)
\psline[linecolor=red,arrows=->,arrowscale=2,linewidth=0.4pt](0.5,-0.867)(1.5,-2.6)(1,-0.55)(2,-0.55)
\pscurve[linewidth=1.2pt](1,-0.55)(1.3,-0.35)(1.8,-0.15)
\psline[linecolor=red,arrows=->,arrowscale=2,linewidth=0.4pt](3,-0.75)(1.5,-1.367)
\psline[linecolor=red,arrows=->,arrowscale=2,linewidth=0.4pt](1.5,-1.367)(-1.5,-2.6)(-3.52,-1.325)(-1.62,-1.325)
\psline[linecolor=red,arrows=->,arrowscale=2,linewidth=0.4pt](3.7,-1.08)(2.233,-1.67)
\psline[linecolor=red,arrows=->,arrowscale=2,linewidth=0.4pt](2.233,-1.67)(-0.7,-2.85)(-3.54,-1.38)(-1.6,-1.38)
\psline[linecolor=red,arrows=->,arrowscale=2,linewidth=0.4pt](4.5,-1.375)(3.3,-1.865)
\psline[linecolor=red,arrows=->,arrowscale=2,linewidth=0.4pt](3.5,-1.783)(1.5,-2.6)(-3.1,-1.6)(-3.56,-1.5)(-1.56,-1.5)
\pscurve[linewidth=1.2pt](-3.51,-1.3)(-3.54,-1.38)(-3.57,-1.56)

}
}
\end{picture}
\caption{The system of a big curve and two small ones transforms two incident parallel flows into two horizontal parallel flows of smaller width.}
\label{fig:systemofcurves}
\end{figure}
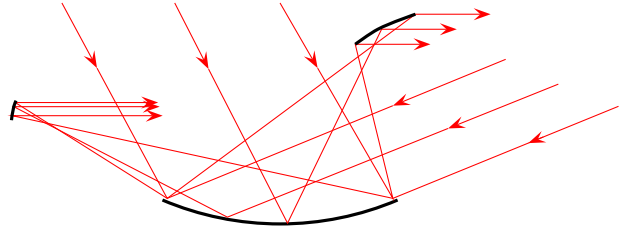

Note in passing that a single incident parallel flow $(n = 1)$ can easily be compressed using two arcs of confocal parabolas, but in the general case of $n \ge 2$ incident flows, parabolas are useless. Instead, we take the big curve to be a circular arc, and carefully design the $n$ small curves.

To make an object invisible, we shield it from the incident flows by adding several collections of curves described above. Each incident flow is reflected by several big curves (circular arcs). The portion of the flow reflected by each arc is then transformed into a parallel flow of smaller width, which goes around the body and is then transformed back into the same portion of the original flow (see Fig.~\ref{fig:attemptinvis}).

\begin{figure}
\begin{picture}(0,130)
\scalebox{1}{
\rput(7,1.8){

\psline(-1.75,-1)(0,-2)(1.75,-1)
\psline(1.75,1)(0,2)(-1.75,1)
\psline[linestyle=dashed,linewidth=0.4pt](-1.75,1)(-1.75,-1)
\psline[linestyle=dashed,linewidth=0.4pt](1.75,1)(1.75,-1)
\pscurve(1.75,-1)(1.65,-0.67)(1.72,-0.33)
\pscurve(1.72,0.33)(1.65,0)(1.72,-0.33)
\pscurve(1.75,1)(1.65,0.67)(1.72,0.33)
\pscurve(-1.75,-1)(-1.65,-0.67)(-1.72,-0.33)
\pscurve(-1.72,0.33)(-1.65,0)(-1.72,-0.33)
\pscurve(-1.75,1)(-1.65,0.67)(-1.72,0.33)
\psline[linecolor=red,arrows=->,arrowscale=1.75,linewidth=0.3pt](-4,-1)(-2.8,-1)
\psline[linecolor=red,arrows=->,arrowscale=1.75,linewidth=0.3pt](-4,-0.67)(-2.8,-0.67)
\psline[linecolor=red,linewidth=0.3pt](-2.8,-1)(-1.75,-1)(-2.25,-0.88)
\psline[linecolor=red,linewidth=0.3pt](-1.65,-0.67)(-2.25,-0.88)
\psline[linecolor=red,arrows=->,arrowscale=1.5,linewidth=0.8pt](-2.25,-0.88)(-2.25,-0.25)
\psline[linecolor=red,arrows=->,arrowscale=1.5,linewidth=0.8pt](-2.25,-0.25)(-2.25,1.3)(0,2.6)
\psline[linecolor=red,arrows=->,arrowscale=1.5,linewidth=0.8pt](0,2.6)(2.25,1.3)(2.25,-0.25)
\psline[linecolor=red,linewidth=0.8pt](2.25,-0.25)(2.25,-0.88)
\psline[linecolor=red,arrows=->,arrowscale=1.75,linewidth=0.3pt](-4,-0.33)(-2.8,-0.33)
\psline[linecolor=red,linewidth=0.3pt](-2.8,-0.33)(-1.75,-0.33)(-2.25,-0.88)
\psline[linecolor=red,linewidth=0.3pt](1.75,-1)(2.25,-0.88)
\psline[linecolor=red,linewidth=0.3pt](1.65,-0.67)(2.25,-0.88)
\psline[linecolor=red,linewidth=0.3pt](1.75,-0.33)(2.25,-0.88)
\psline[linecolor=red,arrows=->,arrowscale=1.75,linewidth=0.3pt](-4,0.33)(-2.8,0.33)
\psline[linecolor=red,arrows=->,arrowscale=1.75,linewidth=0.3pt](-4,0)(-2.8,0)
\psline[linecolor=red,linewidth=0.3pt](-2.8,0.33)(-1.75,0.33)(-2.1,-0.15)
\psline[linecolor=red,linewidth=0.3pt](-1.65,0)(-2.1,-0.15)
\psline[linecolor=red,arrows=->,arrowscale=1.5,linewidth=0.8pt](-2.1,-0.15)(-2.1,0.4)
\psline[linecolor=red,arrows=->,arrowscale=1.5,linewidth=0.8pt](-2.1,0.4)(-2.1,1.2)(0,2.4)
\psline[linecolor=red,arrows=->,arrowscale=1.5,linewidth=0.8pt](0,2.4)(2.1,1.2)(2.1,0.4)
\psline[linecolor=red,linewidth=1pt](2.1,0.4)(2.1,-0.15)
\psline[linecolor=red,linewidth=0.3pt](-1.75,-0.33)(-2.1,-0.15)
\psline[linecolor=red,linewidth=0.3pt](1.75,0.33)(2.1,-0.15)
\psline[linecolor=red,linewidth=0.3pt](1.65,0)(2.1,-0.15)
\psline[linecolor=red,linewidth=0.3pt](1.75,-0.33)(2.1,-0.15)
\psline[linecolor=red,arrows=->,arrowscale=1.75,linewidth=0.3pt](-4,1)(-2.8,1)
\psline[linecolor=red,arrows=->,arrowscale=1.75,linewidth=0.3pt](-4,0.67)(-2.8,0.67)
\psline[linecolor=red,arrows=->,arrowscale=1.75,linewidth=0.3pt](-4,0.33)(-2.8,0.33)
\psline[linecolor=red,linewidth=0.3pt](-2.8,0.33)(-1.75,0.33)(-1.95,0.75)
\psline[linecolor=red,linewidth=0.3pt](-2.8,1)(-1.75,1)(-1.95,0.75)
\psline[linecolor=red,linewidth=0.3pt](-1.65,0.67)(-1.95,0.75)
\psline[linecolor=red,arrows=->,arrowscale=1.5,linewidth=0.8pt](-1.95,0.75)(-1.95,1.1)(0,2.2)
\psline[linecolor=red,arrows=->,arrowscale=1.5,linewidth=0.8pt](0,2.2)(1.95,1.1)(1.95,0.75)
\psline[linecolor=red,linewidth=0.3pt](1.75,1)(1.95,0.75)
\psline[linecolor=red,linewidth=0.3pt](1.65,0.67)(1.95,0.75)
\psline[linecolor=red,linewidth=0.3pt](1.75,0.33)(1.95,0.75)
\psline[linecolor=red,arrows=->,arrowscale=1.75,linewidth=0.3pt](1.75,-1)(4,-1)
\psline[linecolor=red,arrows=->,arrowscale=1.75,linewidth=0.3pt](1.75,-0.33)(4,-0.33)
\psline[linecolor=red,arrows=->,arrowscale=1.75,linewidth=0.3pt](1.75,1)(4,1)
\psline[linecolor=red,arrows=->,arrowscale=1.75,linewidth=0.3pt](1.75,0.33)(4,0.33)
\psline[linecolor=red,arrows=->,arrowscale=1.75,linewidth=0.3pt](2.4,0.67)(4,0.67)
\psline[linecolor=red,arrows=->,arrowscale=1.75,linewidth=0.3pt](2.4,-0.67)(4,-0.67)
\psline[linecolor=red,arrows=->,arrowscale=1.75,linewidth=0.3pt](2.4,0)(4,0)
}
}
\end{picture}
\caption{An attempt to create an invisible body. Only a part of an incident flow is shown, and only two sides of the polygon (the original body) are substituted by hollows in the picture.}
\label{fig:attemptinvis}
\end{figure}
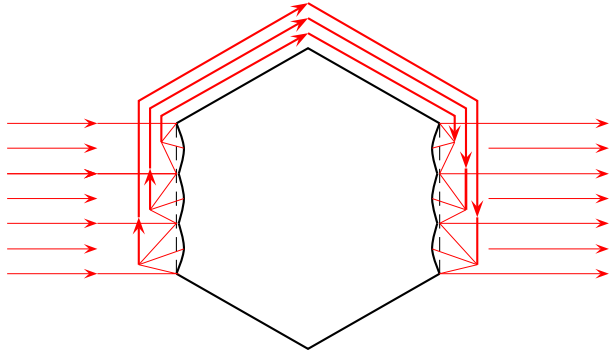

Unfortunately, the added objects (collections of curves) are themselves not invisible. In other words, the resulting body remains visible, but its "visibility"{} is significantly reduced. This observation suggests an idea to define an infinite procedure. The first step of the procedure has just been done. At the second step we make the added objects invisible by adding some more collections of curves (of the second order), and so on. As a result of infinitely many steps we obtain an (infinitely connected) body which is completely invisible.

Of course the devil is in the details, and the rest of the paper is devoted to a rigorous implementation of the above idea. Some details of this simple scheme will be further modified; for instance, we shall actually perform {\it double compression} of incident flows. (For this purpose, the number of small curves will be triplicated.)

First we prove two lemmas.

\section{Two lemmas on flow compression}

In what follows, by {\it light ray} we mean a directed segment of a billiard trajectory between two consecutive reflections. In particular, an {\it incident light ray} is a directed half-line representing the part of a billiard trajectory before the first reflection.

Consider a vertical flow of light rays incident on the graph of a  $C^3$ function $y = \phi(x)$ defined on a segment $[-x_0,\, x_0]$ and such that $\phi''(x) > 0$. (Thus, $\phi$ is strictly convex.) Note that the reflected flow forms a caustic, and the light ray reflected by the graph at $(x, \phi(x))$ will then touch the caustic at the point $L_x(1)$, where
\beqo\label{eq:caustic}
L_x(\tau) = (x, \phi(x)) + \frac{\tau}{2\phi''(x)} (-2\phi'(x),\, 1 - \phi'^2(x)),
\eeqo
and in particular (using the shorthand notation $\phi(0) = \phi_0,\, \phi'(0) = \phi'_0,\, \phi''(0) = \phi''_0$), the light ray reflected at the point $(0, \phi_0)$ will touch the caustic at the point $L_0(1) = (0, \phi_0) + \frac{1}{2\phi''_0} (-2\phi'_0,\, 1 - \phi'^2_0)$.

Indeed, the tangency point is also the point of intersection of two infinitesimally close reflected rays
$$
(x, \phi(x)) + t (-2\phi'(x),\, 1 - \phi'^2(x)) \ \, \text{and} \ \,(x+dx, \phi(x+dx)) + (t+dt) (-2\phi'(x+dx),\, 1 - \phi'^2(x+dx)).
$$
Equating these two expressions and neglecting terms of higher orders, we obtain that $t = 1/(2\phi''(x))$, and therefore the tangency point is $L_x(1)$.

Consider the part of the graph corresponding to $-\ve \le x \le \ve$, where $\ve > 0$ is a small parameter. The length of the corresponding caustic is $O(\ve)$, and the maximum angle between tangent lines at different points of the caustic is also $O(\ve)$. Therefore the width of the reflected flow behind the caustic and at a distance $O(\ve)$ from the caustic is $O(\ve^2)$. Thus, one can expect that a curve crossing this flow (behind the caustic and at a distance $O(\ve)$ from it) in a required direction has the length $O(\ve^2)$.

Formalizing these ideas, we get the following lemma.

\begin{lemma}\label{l1}
There is a real value $c_0 > 0$ such that for all $l > c_0$, in an $O(\ve^2)$-neighborhood of the point $L_0(1 + l\ve)$ there is a smooth curve of length $O(\ve^2)$ such that a vertical flow of particles with the $x$-coordinate $-\ve \le x \le \ve$, after two reflections from the graph of $\phi$ and from this curve is transformed into a parallel flow of width $O(\ve^2)$ in the direction $(1, \phi'_0)$ (see Fig.~\ref{fig:lem1}). Given a constant $c > c_0$, the above estimates $O(\ve^2), \ \ve \to 0$ are uniform in the interval $l \in [c_0,\, c]$. The second reflection induces a smooth one-to-one correspondence between incident light rays and points of the curve.
\end{lemma}

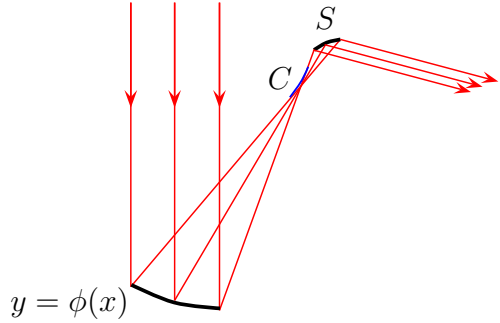
\begin{figure}
\begin{picture}(0,120)
\scalebox{1.4}{
\rput(6.25,5){

\psline[linecolor=red,linewidth=0.4pt,arrows=->,arrowscale=1.5](-0.868,-2)(-0.868,-3)
\psline[linecolor=red,linewidth=0.4pt](-0.868,-2)(-0.868,-4.924)
\psline[linecolor=red,linewidth=0.4pt](-0.868,-4.924)(0.03,-2.45)
\psline[linecolor=red,linewidth=0.4pt,arrows=->,arrowscale=1.5](-1.71,-2)(-1.71,-3)
\psline[linecolor=red,linewidth=0.4pt](-1.71,-2)(-1.71,-4.7)
\psline[linecolor=red,linewidth=0.4pt](-1.71,-4.7)(0.28,-2.35)
\psline[linecolor=red,linewidth=0.4pt,arrows=->,arrowscale=1.5](-1.294,-2)(-1.294,-3)
\psline[linecolor=red,linewidth=0.4pt](-1.294,-2)(-1.294,-4.83)
\psline[linecolor=red,linewidth=0.4pt](-1.294,-4.83)(0.15,-2.38)
\pscurve[linewidth=0.5pt,linecolor=blue](-0.2,-2.9)(-0.0865,-2.7385)(-0.026,-2.6105)
\pscurve[linewidth=1pt](-0.863,-4.906)(-1.299,-4.85)(-1.705,-4.682)
\pscurve[linewidth=1pt](0.03,-2.45)(0.15,-2.38)(0.28,-2.35)
\psline[linecolor=red,linewidth=0.4pt,arrows=->,arrowscale=1.4](0.28,-2.35)(1.78,-2.75)
\psline[linecolor=red,linewidth=0.4pt,arrows=->,arrowscale=1.4](0.02,-2.45)(1.52,-2.85)
\psline[linecolor=red,linewidth=0.4pt,arrows=->,arrowscale=1.4](0.13,-2.4)(1.63,-2.8)
}
}
\rput(5.5,0.2){$y = \phi(x)$}
\rput(8.3,3.2){$C$}
\rput(8.9,4){$S$}
\end{picture}
\caption{The vertical flow is reflected by the graph $y = \phi(x)$. The reflected flow touches the caustic $C$; then it is reflected once again by a smaller curve $S$ and is transformed into another (thinner) parallel flow.}
\label{fig:lem1}
\end{figure}

Let us comment this lemma. The crossing curve should not intersect the caustic; otherwise it forms an undesirable cusp at the point of intersection. By choosing an appropriate $c_0 > 0$, one ensures that the curve is situated behind the caustic at a safe distance, so that non-intersection is guaranteed. By fixing $c > c_0$, one ensures that this distance is at most $O(\ve)$.

In the particular case when $\phi$ is a polynomial of the second order (and therefore its graph is an arc of a parabola with vertical axis), the caustic degenerates to a point --- the focus of the parabola. In this case the width of the reflected flow is proportional to the distance from the focus to the second reflecting curve (which is also an arc of a parabola) and can be made arbitrarily small.

\begin{proof}
The particle reflected at $(x, \phi(x))$ further moves with velocity $(1 + \phi'^2(x))^{-1} (-2\phi'(x),\, 1 - \phi'^2(x))$ and spends the time $t(1 + \phi'^2(x))$ to reach the point
$$
W(x,t) = (x, \phi(x)) + t(-2\phi'(x),\, 1 - \phi'^2(x)).
$$
The partial derivatives of $W$ are $W'_x = (1 - 2t \phi''(x))\, (1,\, \phi'(x))$ and $W'_t = (-2\phi'(x),\, 1 - \phi'^2(x))$, and their vector product is
$$
W'_x \times W'_t = (1 - 2t \phi''(x)) (1 + \phi'^2(x)).
$$
The condition $W'_x \times W'_t = 0$ means that the vectors $W'_x$ and $W'_t$ are linearly dependent and is equivalent to the equation $t = 1/(2\phi''(x))$ determining the caustic. Outside the caustic the map $(x,t) \mapsto W(x,t)$ is a local diffeomorphism.

For the sake of simplicity of notation, define
$$
\tau_1(x) = \frac{(-2\phi'(x),\, 1 - \phi'^2(x))}{1 + \phi'^2(x)}, \quad \tau_2(x) = \frac{(1,\, \phi'(x))}{\sqrt{1 + \phi'^2(x)}}, \quad
\tau_2^0 = \tau_2(0) = \frac{(1,\, \phi'_0)}{\sqrt{1 + \phi'^2_0}}.
$$
We are looking for a reflecting curve in the form $W(x, t(x)), \ x \in [-\ve,\, \ve]$. The condition stating that the vector $\tau_1$, after a reflection from the curve at the point $W(x, t(x))$ becomes the vector $\tau_2^0$, means that the tangent vector to the curve
\beq\label{eq:differen}
\frac{d}{dx} W(x, t(x)) = W'_x + W'_t\, t'(x) = (1 - 2t \phi''(x)) \sqrt{1 + \phi'^2(x)} \tau_2(x) + t'(x) (1 + \phi'^2(x)) \tau_1(x)
\eeq
(a) is parallel to the sum $\tau_1(x) + \tau_ 2^0$ and (b) does not vanish.

The vector $\tau_2^0$ can be decomposed in the basis $\tau_1(x)$,\, $\tau_2(x)$ as follows:
$$
\tau_2^0 = \al(x) \tau_1(x) + (1 + \bt(x)) \tau_2(x),
$$
where $\al(x)$ and $\bt(x)$ are smooth function satisfying $\al(0) = 0 = \bt(0)$. Then the above conditions (a) and (b) take, respectively, the form of the ODE for the unknown function $t = t(x)$
\beq\label{eq:ode}
t' = \frac{1 + \al(x)}{1 + \bt(x)}\ \frac{1 - 2\phi''(x)\, t}{\sqrt{1 + \phi'^2(x)}}, \quad x \in [-\ve,\, \ve]
\eeq
and of the inequality
\beq\label{ineq}
t \ne \frac{1}{2\phi''(x)}.
\eeq
The solution of the (linear inhomogeneous of the first order) differential equation (\ref{eq:ode}) exists for all initial conditions $t(-\ve)$, provided that $1 + \bt(x) \ne 0$ in $[-\ve,\, \ve]$. It remains to check inequality (\ref{ineq}).

Take a value $c_1 > 0$ such that for $\ve > 0$ sufficiently small and for $x \in [-\ve,\, \ve]$ holds
$$
\frac{\phi_0''}{1 + c_1\ve} < \phi''(x) < \phi_0'' (1 + c_1\ve)
$$
(here we use that $\phi$ is in the class $C^3$), and take arbitrary $c$ and $c_0$ such that $c_1 < c_0 < c$. Additionally, for $\ve$ sufficiently small and for $x \in [-\ve,\, \ve]$ we have
$$
0 < \frac{1 + \al(x)}{1 + \bt(x)} < 2 \quad \text{and} \quad \sqrt{1 + \phi'^2(x)} < 2\sqrt{1 + \phi'^2_0}.
$$
Choose an arbitrary $l$ satisfying $c_0 \le l \le c$ and take the initial value of our ODE (\ref{eq:ode})
$$
t(-\ve) = \frac{1 + l\ve}{2\phi_0''}.
$$
Then for $\ve$ sufficiently small we have
$$
1 - 2\phi''(-\ve)t(-\ve) < 1 - \phi''(-\ve)\, \frac{1 + c_1\ve}{\phi_0''} < 0.
$$

Let us prove by {\it reductio ad absurdum} that
\beq\label{ineq-}
1 - 2\phi''(x)t(x) < 0 \quad \text{for all} \ \, x \in [-\ve,\, \ve],
\eeq
and therefore, inequality (\ref{ineq}) is always satisfied. Assume that (\ref{ineq-}) is violated for some $x$ and take $x^* = \inf \{ x \in [-\ve,\, \ve] : 1 - 2\phi''(x)t(x) \ge 0 \}$. We have $-\ve < x^* \le \ve$ and
\beq\label{eq:contr}
1 - 2\phi''(x^*)t(x^*) = 0.
\eeq
For $x < x^*$ one has $t'(x) < 0$, hence $t(x) \le t(-\ve) < (1 + c\ve)/(2\phi_0'')$. Therefore
$$
t'(x) > 2[1 - 2\phi''(x)t] > 2\, \Big[1 - 2\phi_0'' (1 + c_1\ve)\ \frac{1 + c\ve}{2\phi_0''} \Big] > -2\ve\, (2c + c^2 \ve),
$$
and so, for $\ve$ sufficiently small
$$
t(x^*) \ge t(-\ve) + 2\ve \inf_{x \in [-\ve,\, x^*]} t'(x) \ge \frac{1 + c_0\ve}{2\phi_0''} - 4\ve^2\, (2c + c^2\ve) > \frac{1 + c_1\ve}{2\phi_0''} > \frac{1}{2\phi''(x^*)}.
$$
These inequalities contradict equation (\ref{eq:contr}); thus, (\ref{ineq-}) is proved and condition (b) is verified.

Now, using the inequalities $-\ve (2c + c^2\ve) < 1 - 2\phi''(x)\, t(x) < 0$ and $-2\ve\, (2c + c^2\ve) < t'(x) < 0$ and equation (\ref{eq:differen}), one obtains
$$
\Big| \frac{d}{dx} W(x, t(x)) \Big| \le |1 - 2\phi''(x)\, t(x)| \sqrt{1 + \phi'^2(x)} + |t'(x)| (1 + \phi'^2(x)) \le 10\ve (2c + c^2\ve)\, (1 + \phi'^2_0),
$$
and therefore,
$$
\frac{d}{dx} W(x, t(x)) = O(\ve).
$$

Thus, the length of the curve $W(x, t(x)), \ x \in [-\ve,\, \ve]$ is $O(\ve^2)$ (more precisely, it does not exceed $20\ve^2 (2c + c^2\ve)\, (1 + \phi'^2_0)$).

We have $t(0) = t(-\ve) + O(\ve^2) = (1 + l\ve)/(2\phi_0'') + O(\ve^2)$, and so, the middle point $W(0,t(0))$ of the curve is
$$
W(0,t(0)) = (0, \phi_0) + \Big( \frac{1 + l\ve}{2\phi''_0} + O(\ve^2) \Big) (-2\phi'_0,\, 1 - \phi'^2_0) = L_0(1 + l\ve) + O(\ve^2);
$$
thus, the entire curve is contained in an $O(\ve^2)$-neighborhood of $L_0(1 + l\ve)$.

For $\ve$ sufficiently small the light rays of the flow, after the reflection from graph$(\phi)$, are not parallel to any of the tangent lines of the curve, and therefore intersect the curve only once. The light ray corresponding to a value $x, \ x \in [-\ve,\, \ve]$ is reflected by the point $W(x, t(x))$ of the curve. Thus, the mapping $x \mapsto W(x, t(x))$ establishes a smooth one-to-one mapping between the incident rays and points of the curve.
\end{proof}

Now apply Lemma \ref{l1} to the case where an inclined parallel flow is reflected by a circular hollow. Let $\al \ (0 < \al < \pi)$ be the angle of inclination (that is, the velocity of the incident flow be $(\cos\al, -\sin\al)$) and the hollow be of the form
\beq\label{hollow}
y = -\sqrt{r^2 - x^2}, \quad -\ve \le x \le \ve.
\eeq
Both incident and reflected light rays can be viewed as chords of a circle of radius $r$. The reflected light ray touches the caustic, and the tangency point divides the corresponding chord in the ratio 1 : 3 (from the point of reflection). This can be easily seen from the following geometric consideration (see Fig.~\ref{fig:circle}).

Two chords corresponding to reflected light rays are resting on two circular arcs with the angular measures related as 1 : 3. Two triangles formed by these chords are similar, with the ratio of their opposite (resting on the corresponding arcs) sides going to 1 : 3 when one of the chords approaches the other one. Hence, the ratio of lengths of the parts of each chord divided by the intersection point has the same limit.

\begin{figure}
\begin{picture}(0,100)
\scalebox{1}{
\rput(7.5,2){

\psarc[linewidth=1.2pt](0,0){2}{-110}{-70}
\pscircle[linewidth=0.2pt](0,0){2}
\pscircle[linewidth=0.6pt,linestyle=dashed](0,-1.5){0.5}
\psline[linecolor=red,linewidth=0.4pt](-1.732,1)(0,-2)(1.732,1)
\psline[linecolor=red,linewidth=0.4pt](-1.813,0.845)(-0.174,-1.99)(1.932,0.518)
\psline[linewidth=0.2pt](-2.5,-2)(2.5,-2)
\psarc[linewidth=0.3pt](-0.17,-2){1}{120}{180}
\rput(-1.35,-1.7){$\al$}
}
}
\end{picture}
\caption{Two nearby parallel light rays reflected by a circular arc.}
\label{fig:circle}
\end{figure}
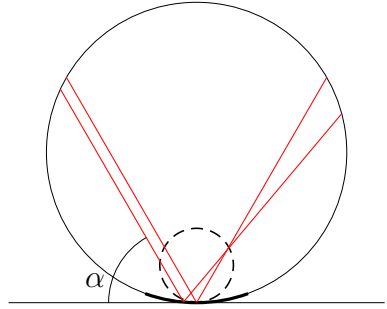

Consider the incident light ray hitting the hollow at the lower point (center of the hollow). The length of the corresponding chord is $2r\sin\al$. Therefore the distance from the point of reflection to the point where the reflected ray touches the caustic is $\frac{1}{2} r\sin\al$. This point lies on the circumference of radius $r/4$ (dashed in Fig.~\ref{fig:circle}) touching the original circumference at the lower point. Parameterize the reflected ray by
\beq\label{Rline}
R(\tau) = R_{\al}(\tau) = (0, -r) + \frac{\tau}{2}\, r\sin\al\, (\cos\al,\, \sin\al);
\eeq
the point of tangency $R(1)$ of this ray to the caustic lies on the aforementioned circumference of radius $r/4$.

The following lemma is a direct consequence of Lemma \ref{l1}.

\begin{lemma}\label{l2}
Fix $r$ and $\al$, and fix the horizontal vector $v_0 = (1,0)$ or $(-1,0)$. There is a real value $c_0 = c_0(r,\al) > 0$ such that for all $l > c_0$, in an $O(\ve^2)$-neighborhood of the point $R(1 + l\ve)$ there exists a smooth curve of length $O(\ve^2)$ such that the flow falling at the angle $\al$ on the hollow (\ref{hollow}), after two reflections, first from the hollow and then from the curve, turns into a horizontal flow of width $O(\ve^2)$ in the direction $v_0$. The tangent to the curve at each point has an inclination angle $\al/2 + O(\ve)$ with respect to the horizontal line, if $v_0 = (1,0)$, and $\al/2 - \pi/2 + O(\ve)$, if $v_0 = (-1,0)$. For any constant $c > c_0$, the above estimates $O(\ve^2)$ are uniform in $l \in [c_0,\, c]$.
\end{lemma}

Notice that this curve (let it be called the {\it second reflecting curve}) is situated at the height $\frac r2 \sin^2\al + O(\ve)$ above the horizontal line $y = -r$ (tangent to the hollow at its lower point).

\begin{proof}
Suppose that $v_0 = (1,0)$. It suffices to take a coordinate system $Oxy$ where the axis $Oy$ is parallel and counter-directional to the incident flow and passes through the center of the hollow (that is, the midpoint of the arc). In this system the function $\phi$ takes the form $\phi(x) = -\sqrt{r^2 - (x - r\cos\al)^2}$.

Applying Lemma \ref{l1} and then passing to the original coordinate system, one comes to the claim of Lemma \ref{l2}. In particular, since at each point of the curve there is a reflection, and the light ray has the inclination $\al + O(\ve)$ before and is horizontal after the reflection, we conclude that the tangent to the curve at this point has the inclination $\al/2 + O(\ve)$.

Note that the second reflecting curve corresponding to the angle of incidence $\al$ and $v_0 = (-1,0)$ is symmetric to the curve corresponding to $\pi - \al$ and $v_0 = (1,0)$, and $R_\al(1 + l\ve)$ is symmetric to $R_{\pi-\al}(1 + l\ve)$, with respect to the vertical line $x = 0$. This implies the claim of the lemma for $v_0 = (-1,0)$.
\end{proof}

\section{Constructing a system of reflecting sets on a fixed level}

Now we are going to incorporate the second reflecting curve into a set of a very special form called {\it s-set} ("s"{} from side). Take again a circular hollow and a parallel flow incident on it at an angle $\al \ (0 < \al < \pi)$. Consider the coordinate system such that the hollow has the form (\ref{hollow}), and correspondingly the flow velocity is $(\cos\al, -\sin\al)$.

We first define an auxiliary object called {\it s-triangle}; it is an isosceles triangle with the horizontal base and the angle $\pi/(2n)$ at the base, and with the apex turned downward. The length of the base is $\ve^{3/2}$ (therefore it is much greater than $\ve^2$ but much smaller than $\ve$). In Figs.~\ref{fig:s-triangle1} and \ref{fig:s-triangle2}, $ABC$ is an s-triangle. Let $MN$ be its midline with the midpoint $P$, and let $Q$ be the midpoint of the base $AC$. We additionally require that the midpoint of $PQ$ coincides with the point $R_\al(1 + l\ve)$. Thus, the s-triangle corresponding to a given incident flow and to a given hollow is uniquely defined by the parameters $\ve,\, r,\, \al$, and $l$, and will further be denoted by $\triangle_{\ve,r,\al,l}$.

By Lemma \ref{l2}, the second reflecting curve corresponding to the parameter $l$ lies in an $O(\ve^2)$-neighborhood of $R_\al(1 + l\ve)$, and therefore belongs to the trapezoid $ACNM$. Its inclination is positive, if the final direction of the flow is $(1,0)$, and negative otherwise. Consider the figure bounded below by the curve and two horizontal segments and above by a part of the broken line $MACN$. The horizontal segments join the endpoints of the curve with the segments $AM$ and $CN$. In Fig.~\ref{fig:s-triangle1} two examples of such a figure, with small and large $\al$, and with the final direction $v_0 = (1,0)$ "to the right"{} are displayed.

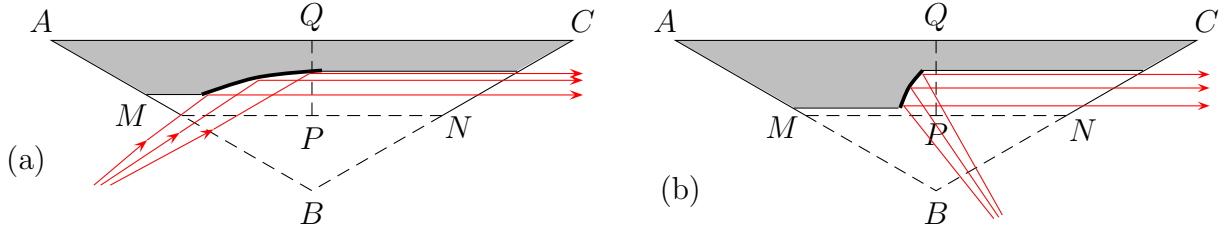
\begin{figure}
\begin{picture}(0,90)
\scalebox{1}{
\rput(3.7,2.5){

\pspolygon[linewidth=0pt,linecolor=white,fillstyle=solid,fillcolor=lightgray]
(0,-0.41)(2.73,-0.41)(3.464,0)(-3.464,0)(-2.2,-0.72)(-1.45,-0.72)(-1,-0.55)(-0.5,-0.48)
\psline[linewidth=0.4pt](-1.9,-0.9)(-3.464,0)(3.464,0)(1.732,-1)
\psline[linewidth=0.4pt,linestyle=dashed](-1.9,-0.9)(0,-2)(1.732,-1)
\psline[linewidth=0.4pt,linestyle=dashed](-1.732,-1)(1.732,-1)
\psline[linewidth=0.4pt,linestyle=dashed](0,-1)(0,0)
   \rput(0,0){\scalebox{0.8}{
   \pscurve[linewidth=1.75pt](-1.9,-0.9)(-1,-0.62)(0.1,-0.5)
   }}
   \psline[linewidth=0.4pt](0,-0.41)(2.73,-0.41)
   \psline[linewidth=0.4pt](-2.2,-0.72)(-1.47,-0.72)
   \rput(-0.01,-0.05){\scalebox{0.75}{
\psline[linecolor=red,linewidth=0.53pt,arrows=->,arrowscale=1.9](-3.63,-2.5)(-1.8,-1.51)
\psline[linecolor=red,linewidth=0.53pt,arrows=->,arrowscale=1.9](-3.8,-2.5)(-2.4,-1.57)
\psline[linecolor=red,linewidth=0.53pt,arrows=->,arrowscale=1.9](-3.92,-2.5)(-2.975,-1.7)
\psline[linecolor=red,linewidth=0.53pt,arrows=->,arrowscale=1.9](-1.8,-1.51)(-0.1,-0.52)(4.75,-0.52)
\psline[linecolor=red,linewidth=0.53pt,arrows=->,arrowscale=1.9](-2.4,-1.57)(-1,-0.64)(4.75,-0.64)
\psline[linecolor=red,linewidth=0.53pt,arrows=->,arrowscale=1.9](-2.975,-1.7)(-1.9,-0.9)(4.75,-0.9)
   }}
\rput(-3.6,0.25){$A$}
\rput(3.6,0.25){$C$}
\rput(0,-2.3){$B$}
\rput(0,-1.3){$P$}
\rput(0,0.3){$Q$}
\rput(-2.4,-1){$M$}
\rput(1.95,-1.15){$N$}
\rput(-3.8,-1.6){(a)}
}
\rput(12,2.5){

\pspolygon[linewidth=0pt,linecolor=white,fillstyle=solid,fillcolor=lightgray]
(-0.2,-0.4)(2.75,-0.4)(3.464,0)(-3.464,0)(-1.89,-0.9)(-0.5,-0.9)(-0.4,-0.72)
\psline[linewidth=0.4pt](-1.732,-1)(-3.464,0)(3.464,0)(1.732,-1)
\psline[linewidth=0.4pt,linestyle=dashed](1.732,-1)(0,-2)(-1.732,-1)
\psline[linewidth=0.4pt,linestyle=dashed](-1.732,-1)(1.732,-1)
\psline[linewidth=0.4pt,linestyle=dashed](0,-1)(0,0)
   \rput(-0.25,0.1){\scalebox{1}{
   \pscurve[linewidth=1.4pt](-0.3,-1)(-0.18,-0.72)(0,-0.5)
   }}
\psline[linewidth=0.4pt](-0.2,-0.4)(2.75,-0.4)
\psline[linewidth=0.4pt](-1.89,-0.9)(-0.5,-0.9)
\rput(0,-1.25){$P$}
      \rput(-0.23,0.05){\scalebox{0.95}{
\psline[linecolor=red,linewidth=0.42pt,arrows=->,arrowscale=1.47](0.98,-2.55)(-0.28,-0.97)(4,-0.97)
\psline[linecolor=red,linewidth=0.42pt,arrows=->,arrowscale=1.47](1.05,-2.53)(-0.18,-0.72)(4,-0.72)
\psline[linecolor=red,linewidth=0.42pt,arrows=->,arrowscale=1.47](1.1,-2.5)(-0.01,-0.53)(4,-0.53)
      }}
\rput(-3.6,0.25){$A$}
\rput(3.6,0.25){$C$}
\rput(0,-2.3){$B$}
\rput(0,0.3){$Q$}
\rput(-2.05,-1.15){$M$}
\rput(1.95,-1.2){$N$}
\rput(-3.4,-2){(b)}
}
}
\end{picture}
\caption{Constructing an s-set: the initial stage. The second reflecting curves (shown in bold) correspond to the cases where (a) $\al$ is small and (b) $\al$ is close to $\pi$. In both cases, the flow behind the caustic is reflected by the curve and turns into a horizontal flow of width $O(\ve^2)$ directed to the right.}
\label{fig:s-triangle1}
\end{figure}

Let us now define the third and fourth reflecting curves. They are arcs of confocal parabolas with the horizontal axis of symmetry. The lower endpoint of the third curve is $N$, and the upper endpoint lies on the horizontal segment to the right of the second curve. The fourth curve is contained in the triangle $MNB$ and is homothetic to the third curve with the homothety center at the common focus and with a small homothety ratio (see Fig.~\ref{fig:s-triangle2}). This pair of curves serves for further compression of the flow; the horizontal flow of width $O(\ve^2)$, after two reflections from these curves, turns into another flow in the same direction (hereafter referred to as the {\it thin flow}) of even smaller width, and the compression ratio coincides with the homothety ratio and can be made arbitrarily small (see Fig.~\ref{fig:s-triangle2}).


Consider two domains: the former (larger) one is the part of the s-triangle $ABC$ bounded by the base $AC$; a part of the segment $AM$; the segment $CN$; the horizontal segment joining the side $AB$ with the lower endpoint of the second reflecting curve; the second and third reflecting curves; and the horizontal segment joining their upper endpoints. The latter (smaller) domain is the curvilinear triangle bounded by the fourth reflecting curve and vertical and horizontal line segments. The union of these domains is called to be an {\it s-set related to the incident flow and to the hollow under consideration}, and $ABC$ is called to be the {\it s-triangle associated with this s-set}.  

\begin{figure}
\begin{picture}(0,110)
\scalebox{1.2}{

\rput(7,2.3){

\pspolygon[linewidth=0pt,linecolor=white,fillstyle=solid,fillcolor=lightgray]
(0,-0.5)(1.52,-0.5)(1.63,-0.75)(1.65,-1.05)(3.464,0)(-3.464,0)(-1.8,-0.96)(-0.3,-0.96)(-0.2,-0.72)
\psline[linewidth=0.4pt](-1.732,-1)(-3.464,0)(3.464,0)(1.732,-1)
\psline[linewidth=0.3pt,linestyle=dashed](1.732,-1)(0,-2)(-1.732,-1)
\psline[linewidth=0.4pt](-1.8,-0.96)(-0.3,-0.96)
\psline[linewidth=0.3pt,linestyle=dashed](0,-2)(0,0)
\pscurve[linewidth=1pt](-0.3,-0.96)(-0.18,-0.72)(0,-0.5)
\psline[linewidth=0.4pt](0,-0.5)(1.52,-0.5)

     \rput(-0.065,-0.025){\scalebox{0.95}{
\psline[linecolor=red,linewidth=0.4pt,arrows=->,arrowscale=1.5](-0.28,-0.97)(0.4,-0.97)
\psline[linecolor=red,linewidth=0.4pt,arrows=->,arrowscale=1.5](-0.18,-0.75)(0.5,-0.75)
\psline[linecolor=red,linewidth=0.4pt,arrows=->,arrowscale=1.5](-0.01,-0.53)(0.6,-0.53)
\psline[linecolor=red,linewidth=0.4pt,arrows=->,arrowscale=1.3](-0.28,-0.97)(1.73,-0.97)(-0.577,-1.41)(2.5,-1.41)
\psline[linecolor=red,linewidth=0.4pt,arrows=->,arrowscale=1.3](-0.18,-0.75)(1.68,-0.75)(-0.56,-1.483)(2.5,-1.483)
\psline[linecolor=red,linewidth=0.4pt,arrows=->,arrowscale=1.3](-0.01,-0.53)(1.57,-0.53)(-0.523,-1.557)(2.5,-1.557)
     }}

\rput(-3.6,0.25){\scalebox{0.8}{$A$}}
\rput(3.6,0.25){\scalebox{0.8}{$C$}}
\rput(0,-2.3){\scalebox{0.8}{$B$}}
\rput(0,0.3){\scalebox{0.8}{$Q$}}

\rput(-0.32,-0.63){\scalebox{0.8}{2}}
\rput(1.75,-0.66){\scalebox{0.8}{3}}
\rput(-0.72,-1.34){\scalebox{0.8}{4}}

     \rput(-0.13,0){\scalebox{1}{
\pscurve[linewidth=1pt](1.7,-1.05)(1.68,-0.75)(1.57,-0.51)
      }}

     \rput(-0.07,0.04){\scalebox{1}{
\psdots[dotsize=1.8pt](0,-1.3)
\pscurve[linewidth=0.83pt](-0.577,-1.39)(-0.56,-1.483)(-0.523,-1.557)
\psline[linewidth=0.83pt](-0.577,-1.39)(-0.577,-1.557)(-0.523,-1.557)
     }}
}
}
\end{picture}
\caption{Constructing an s-set: the final stage. Reflections of the flow from the third and fourth curves (arcs of parabolas) are shown. The second, third, and fourth reflecting curves are labeled by "2"{}, "3"{}, and "4". The s-set is the union of two domains; the former one is shown light gray and the latter one is the tiny spot near the mark "4".}
\label{fig:s-triangle2}
\end{figure}
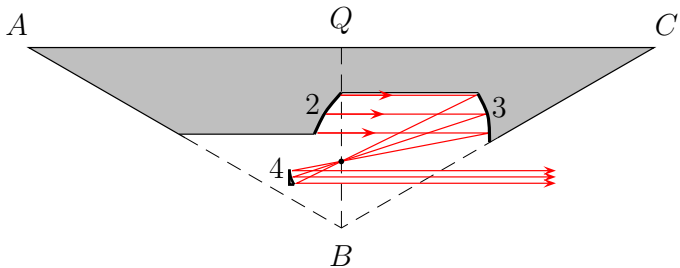

Parameterize the segment $BP$ by $B_\tau$,\, $\tau \in [0,\, 1]$, so as $B_0 = B$ and $B_1 = P$. From the above description we deduce that for any interval $(\tau_1,\, \tau_2) \subset [0,\, 1]$, the resulting thin horizontal flow can be made to pass between the points $B_{\tau_1}$ and $B_{\tau_2}$ (that is, at any height and with arbitrarily small width). Let the common focus of the parabolas be the midpoint of the interval $B_{\tau_1}B_{\tau_2}$, and let the second arc of parabola (the fourth reflecting curve) be chosen to be the maximal one, provided that the resulting thin flow goes between $B_{\tau_1}$ and $B_{\tau_2}$ and the smaller domain is contained in the s-triangle. Thus, the resulting s-set is uniquely determined by the parameters $\ve, \, r, \, \al, \, l,\, v_0 = (\pm 1, 0)$, and $(\tau_1,\, \tau_2)$. 

The above conclusions can be summarized in the following lemma.  

\begin{lemma}\label{l3}
Take an interval $(\tau_1,\, \tau_2) \subset [0,\, 1]$ and fix $v_0 \in \{ (1,0),\, (-1,0) \}$. There is a real value $c_0 = c_0(r,\al) > 0$ such that for all $l > c_0$, in an $O(\ve^2)$-neighborhood of the point $R_\al(1 + l\ve)$ there is an s-set such that the flow falling at the angle $\al$ on the hollow (\ref{hollow}), after a reflection from the hollow and three reflections from the s-set turns into a thin horizontal flow in the direction $v_0$ contained between the points $B_{\tau_1}$ and $B_{\tau_2}$ (see Fig.~\ref{fig:l3}). For any fixed $c > c_0$ the above estimate $O(\ve^2)$ is uniform in $l \in [c_0,\, c]$. The s-triangle $\triangle_{\ve,r,\al,l}$ associated with the s-set does not depend on the choice of $\tau_1$,\, $\tau_2$, and $v_0$.
\end{lemma}

\begin{figure}
\begin{picture}(0,130)
\scalebox{2.4}{
\rput(3,0.15){
\rput(1.35,0){
\psline[linecolor=red,linewidth=0.2pt,arrows=->,arrowscale=1.5](-4.86,1.5)(-4.12,1)
\psline[linecolor=red,linewidth=0.2pt](-4.86,1.5)(-2.55,-0.05)(-0.3,0.92)
\psline[linecolor=red,linewidth=0.2pt,arrows=->,arrowscale=1.5](-4.6,1.5)(-3.87,1)
\psline[linecolor=red,linewidth=0.2pt](-4.6,1.5)(-2.22,-0.12)(-0.3,0.92)
\psline[linecolor=red,linewidth=0.2pt,arrows=->,arrowscale=1.5](-4.18,1.5)(-3.43,1)
\psline[linecolor=red,linewidth=0.2pt](-4.18,1.5)(-1.75,-0.1)(-0.3,0.92)
\psline[linecolor=red,linewidth=0.2pt,arrows=->,arrowscale=1.5](-3.8,1.5)(-3.05,1)
\psline[linecolor=red,linewidth=0.2pt](-3.8,1.5)(-1.486,-0.03)(-0.3,0.92)
     \pspolygon[linewidth=0pt,linecolor=white,fillstyle=solid,fillcolor=white](-0.2,0.8)(-0.2,1)(-0.41,1)(-0.31,0.8)
\rput(-0.6,1.1){
\scalebox{7}{
\pspolygon[linewidth=0pt,linecolor=white,fillstyle=solid,fillcolor=lightgray](0,-0.05)(0.086,0)(-0.086,0)
\psline[linewidth=0.08pt](0.086,0)(-0.086,0)
\psline[linewidth=0.01pt](0.086,0)(0,-0.05)(-0.086,0)}
}
\psline[linecolor=red,linewidth=0.2pt,arrows=->,arrowscale=1.25](-0.1,0.92)(1.3,0.92)
\psline[linecolor=red,linewidth=0.2pt,arrows=->,arrowscale=1.25](-0.1,0.82)(1.3,0.82)
\psline[linecolor=red,linewidth=0.08pt](-0.1,0.853)(1.2,0.853)
\psline[linecolor=red,linewidth=0.08pt](-0.1,0.887)(1.2,0.887)
\psdots[dotsize=1pt](-0.1,0.82)(-0.1,0.92)
\rput(0,0.71){\scalebox{0.4}{$\tau_1$}}
\rput(-0.21,0.98){\scalebox{0.4}{$\tau_2$}}
}
\rput(1.3,1.3){\scalebox{0.4}{T}}
\psarc[linewidth=0.55pt](-0.62,1.5){1.65}{-115}{-65}
}
}
\end{picture}
\caption{A general scheme of reflection from an s-set. Only the s-triangle $T$ containing this set is shown, and three reflections from the s-set are not seen here. The points $B_{\tau_1}$ and $B_{\tau_2}$ are substituted with $\tau_1$ and $\tau_2$.}
\label{fig:l3}
\end{figure}
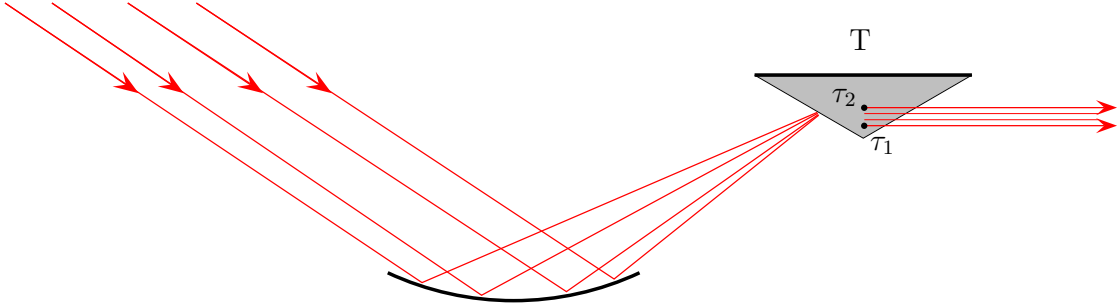

Now fix $r$ and take two complementary angles of incidence $\al$ and $\pi - \al$\, ($0 < \al < \pi/2$).
Consider two parallel flows incident on the infinite periodic sequence of hollows
\beq\label{sequence of hollows}
y = -\sqrt{r^2 - (x - 2m\ve)^2}, \ 2m\ve - \ve \le x \le 2m\ve + \ve, \ m \in \ZZZ
\eeq
at the angles $\al$ and $\pi - \al$. Fix the parameter $l$; then we have two periodic sequences of s-triangles at the same height $R_\al(1 + l\ve) = R_\al(1) + O(\ve)$,
\beqo\label{s-tr}
m_1(2\ve, 0) + \triangle_{\ve,r,\al,l}, \ \, m_1 \in \ZZZ \quad \text{and} \quad m_2(2\ve, 0) + \triangle_{\ve,r,\pi-\al,l},\, \ m_2 \in \ZZZ.
\eeqo
Let us show that for a certain $l = l(\ve)$ these triangles do not overlap.

Introduce the {\it axes} of the portions of the two flows reflected by the $m$th hollow. These axes are, by definition, the half-lines
$$m(2\ve, 0) + R_\al(\tau),\, \tau > 0 \quad \text{and} \quad m(2\ve, 0) + R_{\pi-\al}(\tau),\, \tau > 0.$$
The point of intersection, $m_1(2\ve, 0) + R_\al(1 + l\ve) = m_2(2\ve, 0) + R_{\pi-\al}(1 + l\ve)$, of two axes related to different flows and different hollows corresponds to the value $l$ (identical for the two axes)
$$
l = l_{m}(\ve) = \frac{2m}{r\cos\al\sin\al} - \frac{1}{\ve}, \quad \text{with} \ \, m = m_2 - m_1 \in \ZZZ
$$
(see Fig.~\ref{fig:intersect}).

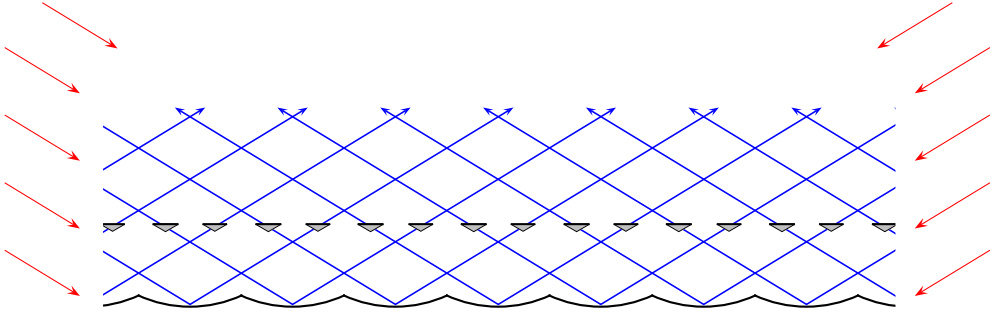
\begin{figure}
\begin{picture}(0,100)

\rput(4,0.15){

\multirput(-6,0)(1.366,0){15}{
\psarc[linewidth=0.8pt](0,1.5){1.65}{-115}{-65}
\psline[linewidth=0.6pt,linecolor=blue]{<->}(-4.3,2.5)(0,-0.12)(4.3,2.5)
\rput(1.5,0.95){
\scalebox{1.9}{
\pspolygon[linewidth=0pt,linecolor=white,fillstyle=solid,fillcolor=lightgray](0,-0.05)(0.086,0)(-0.086,0)
\psline[linewidth=0.4pt](0.086,0)(-0.086,0)
\psline[linewidth=0.1pt](0.086,0)(0,-0.05)(-0.086,0)}}
\rput(-1.9,0.95){
\scalebox{2}{
\pspolygon[linewidth=0pt,linecolor=white,fillstyle=solid,fillcolor=lightgray](0,-0.05)(0.086,0)(-0.086,0)
\psline[linewidth=0.4pt](0.086,0)(-0.086,0)
\psline[linewidth=0.1pt](0.086,0)(0,-0.05)(-0.086,0)}}
}
\rput(-4.8,-0.1){
\pspolygon[linewidth=0pt,linecolor=white,fillstyle=solid,fillcolor=white](-3,-0.1)(3.1,-0.1)(3.1,2.7)(-3,2.7)}
\rput(12,-0.1){
\pspolygon[linewidth=0pt,linecolor=white,fillstyle=solid,fillcolor=white](-3.15,-0.1)(6,-0.1)(6,2.7)(-3.15,2.7)}
\multirput(-3,0.6)(0,0.9){4}{
\psline[linewidth=0.4pt,linecolor=red,arrows=->,arrowscale=1.5](0,0)(1,-0.609)}
\multirput(10.1,0.6)(0,0.9){4}{
\psline[linewidth=0.4pt,linecolor=red,arrows=->,arrowscale=1.5](0,0)(-1,-0.609)
}
\rput(-3,4.2){
\psline[linewidth=0.4pt,linecolor=red,arrows=->,arrowscale=1.5](0.5,-0.305)(1.5,-0.914)}
\rput(10.1,4.2){
\psline[linewidth=0.4pt,linecolor=red,arrows=->,arrowscale=1.5](-0.5,-0.305)(-1.5,-0.914)}

}
\end{picture}
\caption{A periodic sequence of hollows; two incident flows; the axes of the reflected flows; and two disjoint periodic sequences of s-triangles. The points of intersection of the axes correspond to the values $l_m(\ve)$.}
\label{fig:intersect}
\end{figure}

Take $c_0$ as in Lemma \ref{l3} and set $c = c_0 + 2/(r\cos\al\sin\al)$; then one can choose a piecewise continuous function $l(\ve) \in [c_0,\, c]$ so that the values $|l(\ve) - l_m(\ve)|,\, m \in \ZZZ$ are separated from zero by a constant, namely
$$
\min_{m \in \ZZZ} |l(\ve) - l_{m}(\ve)| \ge \frac{1}{r\cos\al\sin\al}.
$$
This choice guarantees that the s-triangles in the two periodic sequences
\beq\label{s-tri}
m_1(2\ve, 0) + \triangle_{\ve,r,\al,l(\ve)}, \ \, m_1 \in \ZZZ \quad \text{and} \quad m_2(2\ve, 0) + \triangle_{\ve,r,\pi-\al,l(\ve)},\, \ m_2 \in \ZZZ
\eeq
are disjoint for $\ve$ sufficiently small (see Fig.~\ref{fig:intersect}).

Now enumerate all s-triangles in both sequences according to $m = 1,\, 2,\ldots$ and take a sequence of horizontal vectors $V = \{ v_0^1,\, v_0^2, \ldots \}$ with $v_0^m = (\pm 1, 0)$, then choose an infinite sequence of disjoint intervals $\mathcal{U} = \{ U^m = (\tau_1^m,\, \tau_2^m),\, m = 1,\,2,\ldots \}$ in $[0,\, 1]$, and to the $m$th triangle ($m = 1,\, 2,\ldots$) assign an s-set so as the $m$th reflected thin flow lies in the interval $U^m$ (more precisely, is contained between the corresponding points $B_{\tau_1^m}$ and $B_{\tau_2^m}$ of the $m$th triangle) and has the direction $v_0^m$; see Fig~\ref{fig:unilevel}. The non-overlapping condition guarantees that each particle makes reflections in only one s-triangle.

\begin{figure}
\begin{picture}(0,155)
\scalebox{3.7}{

\rput(1,0.25){

\multirput(-6,0)(1.366,0){15}{
\psarc[linewidth=0.25pt](0,1.5){1.65}{-115}{-65}
\rput(1.5,1){
\scalebox{3}{
\pspolygon[linewidth=0pt,linecolor=white,fillstyle=solid,fillcolor=lightgray](0,0)(0,-0.035)(5,-0.035)(5,0)
\pspolygon[linewidth=0pt,linecolor=white,fillstyle=solid,fillcolor=gray](0.086,0)(0.043,-0.035)(-0.043,-0.035)(-0.086,0) 
\pspolygon[linewidth=0.005pt](0.086,0)(-0.086,0)(0,-0.07)
}}
\rput(-1.9,1){
\scalebox{3}{
\pspolygon[linewidth=0pt,linecolor=white,fillstyle=solid,fillcolor=gray](0.086,0)(0.043,-0.035)(-0.043,-0.035)(-0.086,0) 
\pspolygon[linewidth=0.005pt](0.086,0)(-0.086,0)(0,-0.07)
}}
}
\rput(-1.75,0.23){\psline[linewidth=0.15pt,linecolor=red,arrows=->,arrowscale=0.8](0,0)(1,-0.35)
\psline[linewidth=0.15pt,linecolor=red,arrows=->,arrowscale=0.8](1,-0.35)(2.025,0.16)
\psline[linewidth=0.15pt,linecolor=red](2.025,0.16)(3.05,0.67)(3.04,0.64)
}
\rput(-1.17,0.23){\psline[linewidth=0.15pt,linecolor=red,arrows=->,arrowscale=0.8](0,0)(1,-0.35)
\psline[linewidth=0.15pt,linecolor=red,arrows=->,arrowscale=0.8](1,-0.35)(1.725,0.16)
\psline[linewidth=0.15pt,linecolor=red](1.725,0.16)(2.45,0.67)(2.45,0.64)
}
\rput(-1.45,0.205){\psline[linewidth=0.15pt,linecolor=red,arrows=->,arrowscale=0.8](0,0)(1,-0.35)
\psline[linewidth=0.15pt,linecolor=red,arrows=->,arrowscale=0.8](1,-0.35)(1.865,0.1675)
\psline[linewidth=0.15pt,linecolor=red,arrows=->,arrowscale=0.8](1.865,0.1675)(2.73,0.685)
}
\rput(1.79,0.02){
\rput(1.75,0.23){\psline[linewidth=0.15pt,linecolor=red,arrows=->,arrowscale=0.8](0,0)(-1,-0.35)
\psline[linewidth=0.15pt,linecolor=red,arrows=->,arrowscale=0.8](-1,-0.35)(-2.025,0.16)
\psline[linewidth=0.15pt,linecolor=red](-2.025,0.16)(-3.05,0.65)(-3.04,0.62)
}
\rput(1.17,0.215){\psline[linewidth=0.15pt,linecolor=red,arrows=->,arrowscale=0.8](0,0)(-1,-0.35)
\psline[linewidth=0.15pt,linecolor=red,arrows=->,arrowscale=0.8](-1,-0.35)(-1.725,0.16)
\psline[linewidth=0.15pt,linecolor=red](-1.725,0.16)(-2.45,0.65)(-2.45,0.62)
}
\rput(1.45,0.2){\psline[linewidth=0.15pt,linecolor=red,arrows=->,arrowscale=0.8](0,0)(-1,-0.35)
\psline[linewidth=0.15pt,linecolor=red,arrows=->,arrowscale=0.8](-1,-0.35)(-1.865,0.1675)
\psline[linewidth=0.15pt,linecolor=red,arrows=->,arrowscale=0.8](-1.865,0.1675)(-2.73,0.67)
}
}
\rput(0.55,0.815){
\pspolygon[linewidth=0pt,linecolor=white,fillstyle=solid,fillcolor=lightgray](-10,0)(-10,0.02)(5,0.02)(5,0) 
\psline[linewidth=0.3pt,linecolor=red,arrows=->,arrowscale=0.8](0,0.01)(-0.3,0.01)
\psline[linewidth=0.3pt,linecolor=red,arrows=->,arrowscale=0.8](-0.3,0.01)(-1.6,0.01)
}
\rput(1.21,0.855){
\pspolygon[linewidth=0pt,linecolor=white,fillstyle=solid,fillcolor=lightgray](-10,0)(-10,0.02)(5,0.02)(5,0) 
\psline[linewidth=0.3pt,linecolor=red,arrows=->,arrowscale=0.8](0,0.01)(0.3,0.01)
\psline[linewidth=0.3pt,linecolor=red,arrows=->,arrowscale=0.8](0.3,0.01)(1.6,0.01)
}
\rput(-4.8,-0.1){
\pspolygon[linewidth=0pt,linecolor=white,fillstyle=solid,fillcolor=white](-3,-0.1)(3.7,-0.1)(3.7,2.7)(-3,2.7)}
\rput(11.95,-0.1){
\pspolygon[linewidth=0pt,linecolor=white,fillstyle=solid,fillcolor=white](-9.1,-0.1)(5.8,-0.1)(5.8,2.7)(-9.1,2.7)}
}
}

\rput(6.2,4.9){1}
\rput(8.7,4.9){2}
\end{picture}
\caption{A collection of hollows and s-triangles: a schematic representation. Only two portions of incident flows at angles $\al$ and $\pi - \al$ are shown. They are reflected first by the corresponding hollows and then by the s-sets 1 and 2. The resulting thin flows are directed to the right and to the left.}
\label{fig:unilevel}
\end{figure}
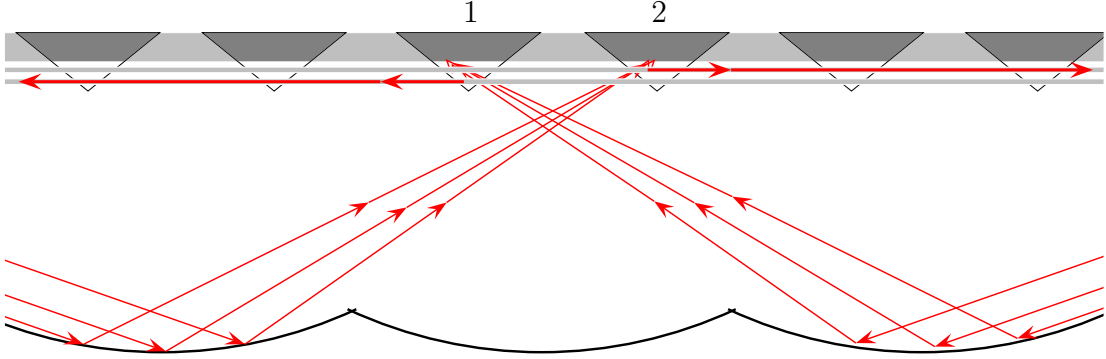

\begin{opr}\rm
The collection of these s-sets and the associated s-triangles (\ref{s-tri}) will be called the $(\ve,r,\al,V,\mathcal{U})$-{\it system of the first kind} (or just an $\al$-{\it system}, if appropriate). In the case $\al = \pi/2$, a $(\pi/2)$-{\it system} is a collection of s-triangles in the periodic sequence
\beq\label{s-tri1}
m(2\ve, 0) + \triangle_{\ve,r,\pi/2,c_0}, \ \, m \in \ZZZ
\eeq
and the corresponding s-sets. Thus, an $\al$-system is defined for all $0 < \al \le \pi/2$.
\end{opr}

Summarizing, for $0 < \al < \pi/2$ we have two sequences of s-sets associated with two $(2\ve, 0)$-periodic sequences of disjoint s-triangles (\ref{s-tri}) with the size $O(\ve^{3/2}), \, \ve \to 0$ and at the same height $R_\al(1) + O(\ve)$. (Notice that $0 < R_\al(1) \le r/2$). If $\al = \pi/2$, we have only one sequence of s-sets associated with the sequence of s-triangles (\ref{s-tri1}) at the height $r/2$. For an incident particle, with the angle of incidence $\al$ or $\pi - \al$, 3 possibilities can be realized.

(a) It is reflected once by a hollow and thrice by an s-set, and then moves horizontally to the right or to the left. All reflected light rays are organized in (infinitely many) horizontal thin flows.   

(b) The particle hits the base of an s-triangle and goes away. The length of each base is $\ve^{3/2}$ and the distance between nearest bases is $O(\ve)$; therefore only a portion $O(\ve^{1/2})$ of the incident flow is reflected this way.

(c) The particle hits a vertex at the base of an s-triangle or an endpoint of a hollow. There are countably many such particles, and they will be neglected in what follows.


\section{Constructing a system of multilevel reflecting sets}

Consider a line segment on the boundary of a half-plane and $n$ parallel flows in this half-plane incident on the segment at the angles $\al_j = \frac{2j - 1}{2n} \pi,\, j = 1,\ldots,n$. Take a regular trapezoid in the half-plane with the height $r$ and with the angle $\pi/(2n)$ at the larger base, and such that the smaller base coincides with the segment. Substitute this segment with a finite sequence of hollows (circular arcs outside the half-plane). Take a coordinate system $xOy$ such that the hollows are described by formula (\ref{sequence of hollows}) and the half-plane takes the form $y \ge -\sqrt{r^2 - \ve^2}$, and assume that the length of the segment is a multiple of $2\ve$.

Take several sequences of horizontal vectors $V_j$ and sequences of disjoint intervals $\mathcal{U}_j$,  
and consider the $(\ve,r,\al_j,V_j,\mathcal{U}_j)$-systems $(j = 1,\ldots,\lfloor\frac{n+1}{2}\rfloor)$ defined in the previous subsection (here $\lfloor\cdots\rfloor$ means the integer part). We take only those s-sets and s-triangles of the systems that correspond to the given finite sequence of hollows covering the segment (for a schematic representation of these s-triangles see Fig.~\ref{fig:block}). Note that the s-triangles lie on heights $\le r/2$ and in the $O(\ve^{3/2})$-neighborhood of the axes of reflected flows; on the other hand, the axes form angles $\ge \pi/(2n)$ with the segment, and therefore, their $O(\ve)$-neighborhoods (at the heights $< r$) lie in the trapezoid. This implies that the s-triangles under the consideration also lie in the trapezoid.


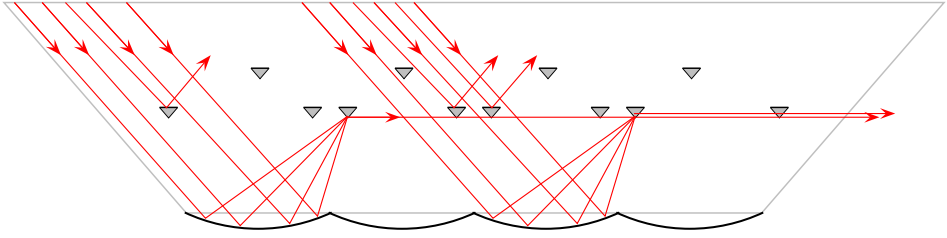
\begin{figure}
\begin{picture}(0,110)
\scalebox{1.4}{

\rput(5.5,0.25){

\pspolygon[linewidth=0.4pt,linecolor=lightgray](2.732,0)(4.464,2)(-4.464,2)(-2.732,0)
\psarc[linewidth=0.55pt](0.683,1.5){1.65}{-115}{-65}
\psarc[linewidth=0.55pt](2.049,1.5){1.65}{-115}{-65}
\psarc[linewidth=0.55pt](-0.683,1.5){1.65}{-115}{-65}
\psarc[linewidth=0.55pt](-2.049,1.5){1.65}{-115}{-65}


\multirput(-1.2,1)(1.366,0){4}
{\pspolygon[linewidth=0pt,linecolor=white,fillstyle=solid,fillcolor=lightgray](0,-0.1)(0.086,0)(-0.086,0)
\psline[linewidth=0.4pt](0.086,0)(-0.086,0)
\psline[linewidth=0.1pt](0.086,0)(0,-0.1)(-0.086,0)}

\multirput(1.2,1)(-1.366,0){4}
{\pspolygon[linewidth=0pt,linecolor=white,fillstyle=solid,fillcolor=lightgray](0,-0.1)(0.086,0)(-0.086,0)
\psline[linewidth=0.4pt](0.086,0)(-0.086,0)
\psline[linewidth=0.1pt](0.086,0)(0,-0.1)(-0.086,0)}

\multirput(-0.07,0)(2.732,0){2}{
\psline[linecolor=red,linewidth=0.3pt,arrows=->,arrowscale=1.5](-4.364,2)(-3.922,1.5)
\psline[linecolor=red,linewidth=0.3pt](-4.364,2)(-2.55,-0.05)(-1.2,0.92)
\psline[linecolor=red,linewidth=0.3pt,arrows=->,arrowscale=1.5](-4.1,2)(-3.655,1.5)
\psline[linecolor=red,linewidth=0.3pt](-4.1,2)(-2.22,-0.12)(-1.2,0.92)
         \psline[linecolor=red,linewidth=0.3pt,arrows=->,arrowscale=1.5](-3.88,2)(-2.92,1)(-2.5,1.5)
\psline[linecolor=red,linewidth=0.3pt,arrows=->,arrowscale=1.5](-3.68,2)(-3.222,1.5)
\psline[linecolor=red,linewidth=0.3pt](-3.68,2)(-1.75,-0.1)(-1.2,0.92)
\psline[linecolor=red,linewidth=0.3pt,arrows=->,arrowscale=1.5](-3.3,2)(-2.852,1.5)
\psline[linecolor=red,linewidth=0.3pt](-3.3,2)(-1.486,-0.03)(-1.2,0.92)}

         \psline[linecolor=red,linewidth=0.3pt,arrows=->,arrowscale=1.5](-0.75,2)(0.17,1)(0.6,1.5)

\psline[linecolor=red,linewidth=0.3pt,arrows=->,arrowscale=1.4](-1.2,0.91)(-0.7,0.91)
\psline[linecolor=red,linewidth=0.3pt,arrows=->,arrowscale=1.4](-1.2,0.91)(3.85,0.91)
\psline[linecolor=red,linewidth=0.3pt,arrows=->,arrowscale=1.4](1.52,0.945)(4,0.945)

\multirput(-2.1,1.375)(1.366,0){4}{
\pspolygon[linewidth=0pt,linecolor=white,fillstyle=solid,fillcolor=lightgray](0,-0.1)(0.086,0)(-0.086,0)
\psline[linewidth=0.4pt](0.086,0)(-0.086,0)
\psline[linewidth=0.1pt](0.086,0)(0,-0.1)(-0.086,0)}

}
}
\end{picture}
\caption{A schematic representation of a collection of s-triangles contained in a trapezoid. Only two portions of a flow incident on the hollows are shown. There are two periodic sequences of s-triangles corresponding to inclined flows and a sequence corresponding to the vertical flow ($\al = \pi/2$), with 4 s-triangles in each sequence.}
\label{fig:block}
\end{figure}

The s-triangles of an $\al_j$-system are situated on the height of $\frac r2\, \sin^2 \al_j + O(\ve)$; therefore the triangles of different systems lie on different heights and do not overlap for $\ve$ sufficiently small.

Consider the finite collection of s-triangles and s-sets in all $(\ve,r,\al_j,V_j,\UUU_j)$-systems, $j = 1,\ldots,\lfloor\frac{n+1}{2}\rfloor$ corresponding to the given finite sequence of hollows. It is still not quite satisfactory for our purposes; let us see why.

For an incident particle at an angle $\al_j$, there are three possibilities. (Here and in the sequel we exclude from consideration finitely many particles that make the first reflection at a vertex of the base of an s-triangle or at an endpoint of a hollow.)

(a) It is reflected once by a hollow and thrice by the s-set corresponding to this hollow and to the given incident flow, and then moves horizontally (to the left or to the right).

(b) It hits the base of an s-triangle. (We do not care about what happens afterwards.)

(c) It is reflected by a hollow and then hits an irrelevant s-set (corresponding to a different flow).

The case (c) is undesirable, since the trajectory after hitting an irrelevant set cannot be controlled. To exclude this case, we need to add some more sequences of triangles (let them call {\it false s-triangles}) shielding the triangles of our collection from irrelevant reflected flows (see Fig.~\ref{fig:shield}). All the triangles of the original collection (that is, those that generate s-sets) will be called {\it true s-triangles}, in order to distinguish them from the false ones.

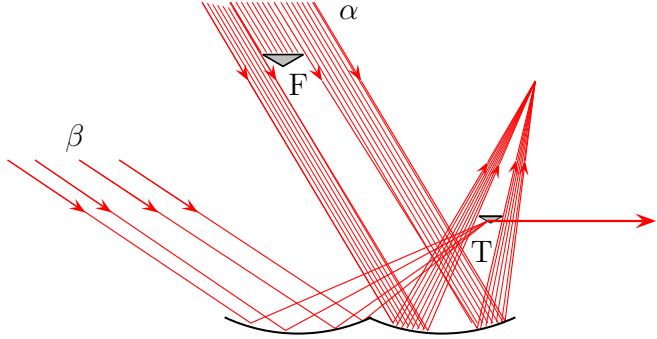
\begin{figure}
\begin{picture}(0,130)
\scalebox{1.4}{

\rput(5.5,0.25){

\rput(0.99,0.965){
\scalebox{1.3}{
\pspolygon[linewidth=0pt,linecolor=white,fillstyle=solid,fillcolor=lightgray](0,-0.05)(0.086,0)(-0.086,0)
\psline[linewidth=0.4pt](0.086,0)(-0.086,0)
\psline[linewidth=0.2pt](0.086,0)(0,-0.05)(-0.086,0)}}

\rput(-1.045,2.5){
\scalebox{2.25}{
\pspolygon[linewidth=0pt,linecolor=white,fillstyle=solid,fillcolor=lightgray](0,-0.05)(0.086,0)(-0.086,0)
\psline[linewidth=0.2pt](0.086,0)(-0.086,0)
\psline[linewidth=0.2pt](0.086,0)(0,-0.05)(-0.086,0)}}

\rput(2.7,0){
\psline[linecolor=red,linewidth=0.3pt,arrows=->,arrowscale=1.5](-4.364,3)(-3.922,2.25)
\psline[linecolor=red,linewidth=0.3pt,arrows=->,arrowscale=1.3](-3.922,2.25)(-2.55,-0.05)(-1.65,1.48)
\psline[linecolor=red,linewidth=0.3pt](-1.65,1.48)(-1.2,2.25)
\psline[linecolor=red,linewidth=0.3pt,arrows=->,arrowscale=1.5](-4.1,3)(-3.655,2.25)
\psline[linecolor=red,linewidth=0.3pt,arrows=->,arrowscale=1.3](-3.655,2.25)(-2.22,-0.12)(-1.54,1.46)
\psline[linecolor=red,linewidth=0.3pt](-1.54,1.46)(-1.2,2.25)
\psline[linecolor=red,linewidth=0.3pt,arrows=->,arrowscale=1.5](-3.68,3)(-3.222,2.25)
\psline[linecolor=red,linewidth=0.3pt,arrows=->,arrowscale=1.3](-3.222,2.25)(-1.75,-0.1)(-1.38,1.5)
\psline[linecolor=red,linewidth=0.3pt](-1.38,1.5)(-1.2,2.25)
\psline[linecolor=red,linewidth=0.3pt,arrows=->,arrowscale=1.5](-3.3,3)(-2.852,2.25)
\psline[linecolor=red,linewidth=0.3pt,arrows=->,arrowscale=1.3](-2.852,2.25)(-1.486,-0.03)(-1.295,1.49)
\psline[linecolor=red,linewidth=0.3pt](-1.295,1.49)(-1.2,2.25)}
\rput(-0.68,2.22){\scalebox{0.7}{F}}
\rput(1.06,0.63){\scalebox{0.7}{T}}
\rput(-0.2,2.9){\scalebox{0.7}{$\al$}}
\rput(-2.8,1.7){\scalebox{0.7}{$\bt$}}

\multirput(2.71,0)(0.05,0){8}{
\psline[linecolor=red,linewidth=0.05pt](-4.1,3)(-3.81,2.52)}

\multirput(3.45,0)(0.05,0){7}{
\psline[linecolor=red,linewidth=0.05pt](-4.364,3)(-2.55,-0.05)}

\multirput(2.7,0)(0.05,-0.01){7}{
\psline[linecolor=red,linewidth=0.05pt](-4.364,3)(-2.55,-0.05)}

\psline[linecolor=red,linewidth=0.05pt](0.25,-0.08)(1.57,2.25)
\psline[linecolor=red,linewidth=0.05pt](0.3,-0.1)(1.57,2.25)
\psline[linecolor=red,linewidth=0.05pt](0.35,-0.1)(1.57,2.25)
\psline[linecolor=red,linewidth=0.05pt](0.4,-0.11)(1.57,2.25)
\psline[linecolor=red,linewidth=0.05pt](0.45,-0.11)(1.57,2.25)
\psline[linecolor=red,linewidth=0.05pt](0.5,-0.12)(1.57,2.25)

\psline[linecolor=red,linewidth=0.05pt](1.05,-0.1)(1.57,2.25)
\psline[linecolor=red,linewidth=0.05pt](1.1,-0.08)(1.57,2.25)
\psline[linecolor=red,linewidth=0.05pt](1.15,-0.06)(1.57,2.25)
\psline[linecolor=red,linewidth=0.05pt](1.2,-0.04)(1.57,2.25)
\psline[linecolor=red,linewidth=0.05pt](1.25,-0.04)(1.57,2.25)

\rput(1.35,0){
\psline[linecolor=red,linewidth=0.3pt,arrows=->,arrowscale=1.5](-4.86,1.5)(-4.12,1)
\psline[linecolor=red,linewidth=0.3pt](-4.86,1.5)(-2.55,-0.05)(-0.3,0.92)
\psline[linecolor=red,linewidth=0.3pt,arrows=->,arrowscale=1.5](-4.6,1.5)(-3.87,1)
\psline[linecolor=red,linewidth=0.3pt](-4.6,1.5)(-2.22,-0.12)(-0.3,0.92)
\psline[linecolor=red,linewidth=0.3pt,arrows=->,arrowscale=1.5](-4.18,1.5)(-3.43,1)
\psline[linecolor=red,linewidth=0.3pt](-4.18,1.5)(-1.75,-0.1)(-0.3,0.92)
\psline[linecolor=red,linewidth=0.3pt,arrows=->,arrowscale=1.5](-3.8,1.5)(-3.05,1)
\psline[linecolor=red,linewidth=0.3pt](-3.8,1.5)(-1.486,-0.03)(-0.3,0.92)
\psline[linecolor=red,linewidth=0.6pt,arrows=->,arrowscale=1.5](-0.3,0.92)(1.3,0.92)
}

\psarc[linewidth=0.55pt](0.683,1.5){1.65}{-115}{-65}
\psarc[linewidth=0.55pt](-0.683,1.5){1.65}{-115}{-65}
}
}
\end{picture}
\caption{Two incident flows and two s-triangles are shown. The true s-triangle T is related to the $\bt$-flow, and is shielded by the false s-triangle F from the $\al$-flow.}
\label{fig:shield}
\end{figure}

The false s-triangles are also isosceles, with the horizontal base and the angle $\pi/(2n)$ at the base, and with the apex turned downward. They should be placed higher than the true ones, that is, on a height in the interval $(r/2,\, r)$. We shall take the height of their bases to be $3r/4$. With all these restrictions, we take a false s-triangle to be the smallest triangle that shields a fixed (true) s-triangle from a fixed piece of flow. This choice uniquely defines all false triangles.  

Note that the portion of the reflected flow corresponding to the angle of incidence $\al_j$ and to a certain hollow can hit at most two triangles of an irrelevant $\al_i$-system, provided that $\frac{n+1}{2} - i > |\frac{n+1}{2} - j|$, and does not hit triangles of this system otherwise. Indeed, this inequality ensures that the triangles of the $\al_i$-system are at a lower height than the triangles of the $\al_j$-system (and in particular $\al_i < \pi/2$). Since the $\al_i$-system is composed of two periodic sequences with the period $2\ve$ and the width of the reflected flow is smaller than $2\ve$, only one triangle of each sequence can be hit.

Thus, the number of false s-triangles on the way of a portion of the $j$th flow does not exceed $n - 1 - |n + 1 - 2j|$ (and therefore is zero for the flows incident at the smallest angles $\pi/(2n)$ and $(2n - 1)\pi/(2n)$). Any such false triangle is a representative of a periodic sequence of false triangles; the number of sequences does not exceed $\lfloor\frac{n+1}{2}\rfloor (\lfloor\frac{n+1}{2}\rfloor - 1)$. Thus, the total number of false triangles is $O(1/\ve)$. Their size is $O(\ve^{3/2})$, and they may intersect each other (but do not intersect true s-triangles).

\begin{opr}\rm
The collection of sets including the finite sequence of hollows resting on the given line segment, the s-sets of all $(\ve,r,\al_j,V_j,\UUU_j)$-systems $(j = 1,\ldots,\lfloor\frac{n+1}{2}\rfloor)$ related to this sequence of hollows, as well as the associated true s-triangles, and the false s-triangles, is called the $[\ve,r,\bbb,V,\UUU]$-{\it system of the second kind induced by the trapezoid}. Here $\bbb$ is the length of the segment, and the collection $V\, (\UUU)$ is composed of all sequences of vectors $V_j$ (sequences of intervals $\UUU_j$).
\end{opr}

Recall that the segment coincides with the smaller base of the trapezoid and its length $b$ is supposed to be a multiple of $2\ve$.


It follows from the construction that a system of the second kind satisfies the following properties.

$\bullet$ The collection of (true and false) s-triangles in the system does not depend on the choice of $V$ and $\UUU$.

$\bullet$ The true s-triangles corresponding to the flow incident at a fixed angle $\al_j$ form a finite $(2\ve,0)$-periodic sequence, and the sequences corresponding to $\al_j$ and $\pi - \al_j$ are mutually symmetric with respect to the perpendicular bisector of the segment (and in particular, the sequence of triangles corresponding to $\pi/2$ is itself symmetric with respect to this bisector). The set of false triangles is also symmetric with respect to the bisector.

$\bullet$ For the incident particles only the cases (a) and (b) can be realized.
Only a part $O(\ve^{1/2})$ of the incident flows satisfies (b).

\section{Constructing a polygonal system of reflecting sets}

In this and the next section we consider $2n$ parallel flows with the directions being the external bisectors of the angles of $P_1$.

Let us give some more definitions. Let $O$ be the center of the polygon $P_1$. The image of $\pl P_1$ under a dilation centered at $O$ is called a {\it polygonal contour}. A domain bounded by two polygonal contours is called a {\it polygonal ring}, or a {\it p-ring}. These contours are called the {\it inner} and {\it outer boundaries} of the p-ring.

Consider the rhombus ($ABCD$ in Fig.~\ref{fig:v-set}\,(a)) bounded by two adjacent sides of the outer polygon of the p-ring and by the extensions of the corresponding sides of the inner polygon. Divide it into two triangles by the larger diagonal, and apply the dilation with the ratio 3 to the outer one, the center of the dilation being the center of the rhombus. The resulting triangle (shaded in Fig.~\ref{fig:v-set}\,(a)) is called a {\it v-triangle} ("v"{} from vertex) associated with the p-ring.

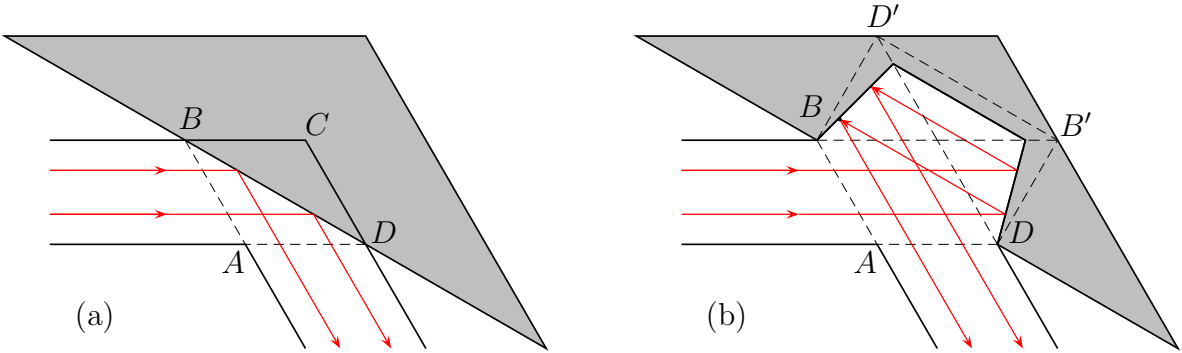
\begin{figure}
\begin{picture}(0,120)
\scalebox{0.8}{
\rput(9.5,1.7){

\rput(-5.5,0){
\pspolygon[fillstyle=solid,fillcolor=lightgray](-4,3.464)(2,3.464)(5,-1.732)
\psline(-3.25,0)(0,0)(1,-1.732)
\psline(-3.25,1.732)(1,1.732)(3,-1.732)
\psline[linewidth=0.4pt,linestyle=dashed](-1,1.732)(0,0)(2,0)
\psline[linecolor=red,linewidth=0.6pt,arrows=->,arrowscale=1.6](-3.25,1.23)(-1.3,1.23)
\psline[linecolor=red,linewidth=0.6pt,arrows=->,arrowscale=1.6](-3.25,0.5)(-1.3,0.5)
\psline[linecolor=red,linewidth=0.6pt,arrows=->,arrowscale=1.6](-3.25,1.23)(-0.134,1.23)(1.577,-1.732)
\psline[linecolor=red,linewidth=0.6pt,arrows=->,arrowscale=1.6](-3.25,0.5)(1.134,0.5)(2.423,-1.732)
\rput(-0.2,-0.3){\scalebox{1.25}{$A$}}
\rput(-0.9,2.1){\scalebox{1.25}{$B$}}
\rput(1.2,2){\scalebox{1.25}{$C$}}
\rput(2.3,0.2){\scalebox{1.25}{$D$}}
\rput(-2.5,-1.2){\scalebox{1.25}{(a)}}
}

\rput(5,0){
\pspolygon[fillstyle=solid,fillcolor=lightgray]
(-4,3.464)(2,3.464)(5,-1.732)(2,0)(2.464,1.732)(0.268,3)(-1,1.732)
\psline(-3.25,0)(0,0)(1,-1.732)
\psline(-3.25,1.732)(-1,1.732)
\psline[linewidth=0.4pt,linestyle=dashed](-1,1.732)(1,1.732)(2,0)
\psline(2,0)(3,-1.732)
\psline[linewidth=0.4pt,linestyle=dashed](-1,1.732)(0,0)(2,0)
\psline[linecolor=red,linewidth=0.6pt,arrows=->,arrowscale=1.6](-3.25,1.23)(-1.3,1.23)
\psline[linecolor=red,linewidth=0.6pt,arrows=->,arrowscale=1.6](-3.25,0.5)(-1.3,0.5)
\psline[linecolor=red,linewidth=0.6pt,arrows=->,arrowscale=1.6](-1.3,0.5)(2.13,0.5)(-0.62,2.08)
\psline[linecolor=red,linewidth=0.6pt,arrows=->,arrowscale=1.6](-3.25,1.23)(2.33,1.23)(-0.1,2.63)
\psline[linecolor=red,linewidth=0.6pt,arrows=->,arrowscale=1.6](-0.62,2.08)(-0.134,1.23)(1.577,-1.732)

\psline[linecolor=red,linewidth=0.6pt,arrows=->,arrowscale=1.6](-0.1,2.63)(2.423,-1.732)
\psline[linewidth=0.4pt,linestyle=dashed](-1,1.732)(0,3.464)(3,1.732)(2,0)
\psline(2,0)(2.464,1.732)(0.268,3)(-1,1.732)
\psline[linewidth=0.4pt,linestyle=dashed](1,1.732)(3,1.732)
\psline[linewidth=0.4pt,linestyle=dashed](1,1.732)(0,3.464)
\rput(-0.2,-0.3){\scalebox{1.25}{$A$}}
\rput(-1.1,2.3){\scalebox{1.25}{$B$}}
\rput(2.4,0.2){\scalebox{1.25}{$D$}}
\rput(0.1,3.8){\scalebox{1.25}{$D'$}}
\rput(3.3,2){\scalebox{1.25}{$B'$}}
\rput(-2.5,-1.2){\scalebox{1.25}{(b)}}
\psdots[dotsize=2pt](-0.62,2.08)
}

}
}
\end{picture}
\caption{A v-triangle and a v-set are shown in figures (a) and (b). Reflection of the flow from a v-set preserves the order of light rays, as shown in figure (b).}
\label{fig:v-set}
\end{figure}

In Figure~\ref{fig:v-set}\,(b), the segments $BB'$ and $DD'$ are extensions of sides of the outer polygon, and the segments $BD'$ and $DB'$ are obviously orthogonal to $BD$. Draw a line through $B$ until the intersection with $DD'$, and a line through $D$ until the intersection with $BB'$, both lines making the angle $\pi/(4n)$ with $BD'$, clockwise and counterclockwise. Joining the points of intersection, as a result we obtain a closed hexagonal broken line. The domain bounded by this line is called a {\it v-set associated with this polygonal ring} (shaded in Fig.~\ref{fig:v-set}\,(b)).

Now consider the union of all v-sets associated with a given p-ring, and consider a particle traveling in the p-ring with velocity parallel to one of its sides. When hitting a v-set, the particle is reflected according to the billiard law (see Fig.~\ref{fig:v-set}\,(b)). One easily sees that the trajectory of the particle is closed and is symmetric with respect to all lines of symmetry of $P_1$.

Take a polygonal contour and a finite collection of disjoint closed line segments $\{ \SSS_\iii \}$ on the contour symmetric with respect to all lines of symmetry of $P_1$. Fix $\ve$ and $r$, and for each $\iii$ take the regular trapezoid (outside the contour) whose smaller base coincides with the $\iii$th segment, $\SSS_\iii$, and with the height $r$ and the angle $\pi/(2n)$ at the larger base (see Fig.~\ref{fig:SegmTrap}).
We choose $r$ small enough, so that the trapezoids are disjoint.
For each $\iii$ take the $[\ve,r,\bbb_{\iii},V_{\iii},\UUU_{\iii}]$-system associated with the $\iii$th trapezoid, where $\bbb_\iii$ is the length of the segment $\SSS_\iii$, and $V_\iii$ and $\UUU_\iii$ will be specified below. We assume that all $\bbb_\iii$ are multiples of $2\ve$.
The collection of s-triangles in all these systems is symmetric with respect to all lines of symmetry of $P_1$.
Substitute each segment $\SSS_\iii$ with the corresponding finite sequence of hollows.

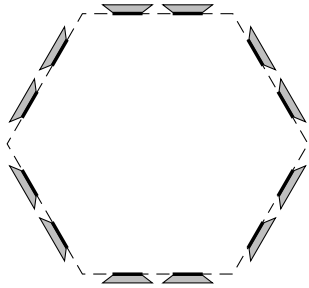
\begin{figure}
\begin{picture}(0,99)
\scalebox{1}{
\rput(7.5,1.7){

\rput{60}(0,0){\pspolygon[linewidth=0.4pt,fillstyle=solid,fillcolor=lightgray](0.6,-1.75)(0.2,-1.75)(0.07,-1.85)(0.73,-1.85)
\pspolygon[linewidth=0.4pt,fillstyle=solid,fillcolor=lightgray](-0.6,-1.75)(-0.2,-1.75)(-0.07,-1.85)(-0.73,-1.85)}
\rput{-60}(0,0){\pspolygon[linewidth=0.4pt,fillstyle=solid,fillcolor=lightgray](0.6,-1.75)(0.2,-1.75)(0.07,-1.85)(0.73,-1.85)
\pspolygon[linewidth=0.4pt,fillstyle=solid,fillcolor=lightgray](-0.6,-1.75)(-0.2,-1.75)(-0.07,-1.85)(-0.73,-1.85)}
\rput{60}(0,0){\pspolygon[linewidth=0.4pt,fillstyle=solid,fillcolor=lightgray](0.6,1.75)(0.2,1.75)(0.07,1.85)(0.73,1.85)
\pspolygon[linewidth=0.4pt,fillstyle=solid,fillcolor=lightgray](-0.6,1.75)(-0.2,1.75)(-0.07,1.85)(-0.73,1.85)}
\rput{-60}(0,0){\pspolygon[linewidth=0.4pt,fillstyle=solid,fillcolor=lightgray](0.6,1.75)(0.2,1.75)(0.07,1.85)(0.73,1.85)
\pspolygon[linewidth=0.4pt,fillstyle=solid,fillcolor=lightgray](-0.6,1.75)(-0.2,1.75)(-0.07,1.85)(-0.73,1.85)}
\pspolygon[linewidth=0.4pt,fillstyle=solid,fillcolor=lightgray](0.6,1.75)(0.2,1.75)(0.07,1.85)(0.73,1.85)
\pspolygon[linewidth=0.4pt,fillstyle=solid,fillcolor=lightgray](-0.6,1.75)(-0.2,1.75)(-0.07,1.85)(-0.73,1.85)
\pspolygon[linewidth=0.4pt,fillstyle=solid,fillcolor=lightgray](0.6,-1.75)(0.2,-1.75)(0.07,-1.85)(0.73,-1.85)
\pspolygon[linewidth=0.4pt,fillstyle=solid,fillcolor=lightgray](-0.6,-1.75)(-0.2,-1.75)(-0.07,-1.85)(-0.73,-1.85)
\pspolygon[linewidth=0.4pt,linestyle=dashed](2,0)(1,1.732)(-1,1.732)(-2,0)(-1,-1.732)(1,-1.732)
\psline[linewidth=1.2pt](-0.6,1.732)(-0.2,1.732)
\psline[linewidth=1.2pt](0.6,1.732)(0.2,1.732)
\psline[linewidth=1.2pt](-0.6,-1.732)(-0.2,-1.732)
\psline[linewidth=1.2pt](0.6,-1.732)(0.2,-1.732)
\psline[linewidth=1.2pt](1.8,0.3464)(1.6,0.6928)
\psline[linewidth=1.2pt](1.4,1.0392)(1.2,1.3856)
\psline[linewidth=1.2pt](1.8,-0.3464)(1.6,-0.6928)
\psline[linewidth=1.2pt](1.4,-1.0392)(1.2,-1.3856)
\psline[linewidth=1.2pt](-1.8,0.3464)(-1.6,0.6928)
\psline[linewidth=1.2pt](-1.4,1.0392)(-1.2,1.3856)
\psline[linewidth=1.2pt](-1.8,-0.3464)(-1.6,-0.6928)
\psline[linewidth=1.2pt](-1.4,-1.0392)(-1.2,-1.3856)

}
}
\end{picture}
\caption{A finite collection of segments (shown in bold) and the corresponding trapezoids.}
\label{fig:SegmTrap}
\end{figure}

Consider the collection of true s-triangles in these systems. There is a natural pairing of triangles (and correspondingly of s-sets). Namely, each true s-triangle corresponds to a certain parallel flow and a certain hollow. Consider the symmetry with respect to a main diagonal of $P_1$. The triangles corresponding to symmetric hollows and to opposite flows orthogonal to this diagonal form a pair (see Fig.~\ref{fig:pairing}). Note that the triangles in a pair are responsible for flows with complementary angles of incidence, say $\al$ and $\pi - \al$, and therefore are situated at the same height above the contour.

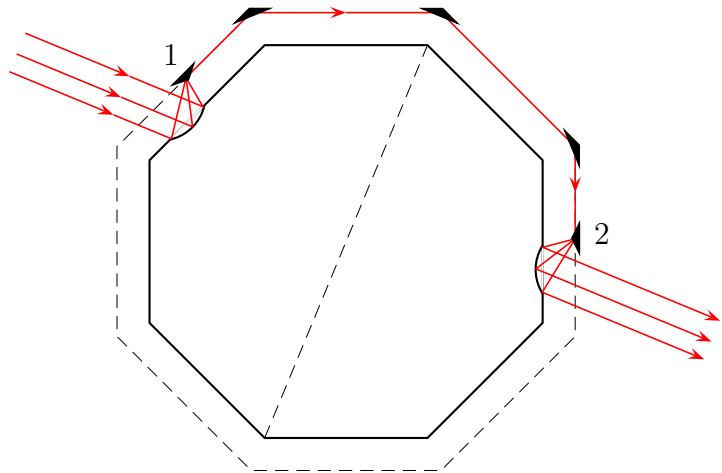
\begin{figure}
\begin{picture}(0,160)

\rput(8,2.75){
\scalebox{1}{

     \rput{-22.5}(0,0){
\pspolygon[linewidth=0.8pt](2.8284,0)(2,2)(0,2.8284)(-2,2)(-2.8284,0)(-2,-2)(0,-2.8284)(2,-2)
\psline[linestyle=dashed,linewidth=0.3pt](0,-2.8284)(0,2.8284)

\psline[linestyle=dashed,linewidth=0.3pt](2.331,2.331)(3.297,0)(2.331,-2.331)(0,-3.297)(-2.331,-2.331)(-3.297,0)(-2.331,2.331)

\psline[linewidth=0.8pt,linecolor=white](2.668,0.375)(2.445,0.93)
\psarc[linewidth=0.8pt](3.025,0.85){0.6}{170}{235}

\psline[linewidth=0.8pt,linecolor=white](-2.668,0.375)(-2.445,0.93)
\psarc[linewidth=0.8pt](-3.025,0.85){0.6}{-55}{10}

\psline[linecolor=red,linewidth=0.6pt,arrows=->,arrowscale=1.6](2.8,1.2)(2.668,0.375)(5,0.375)
\psline[linecolor=red,linewidth=0.6pt,arrows=->,arrowscale=1.6](2.8,1.2)(2.445,0.93)(5,0.93)
\psline[linecolor=red,linewidth=0.6pt,arrows=->,arrowscale=1.6](2.566,1.766)(2.8,1.2)(2.47,0.63)(5,0.63)

\psline[linecolor=red,linewidth=0.6pt,arrows=->,arrowscale=1.6](-5,0.375)(-3.5,0.375)
\psline[linecolor=red,linewidth=0.6pt](-3.5,0.375)(-2.668,0.375)(-2.8,1.2)
\psline[linecolor=red,linewidth=0.6pt,arrows=->,arrowscale=1.6](-5,0.93)(-3.5,0.93)
\psline[linecolor=red,linewidth=0.6pt](-3.5,0.93)(-2.445,0.93)(-2.8,1.2)
\psline[linecolor=red,linewidth=0.6pt,arrows=->,arrowscale=1.6](-5,0.63)(-3.5,0.63)
\psline[linecolor=red,linewidth=0.6pt,arrows=->,arrowscale=1.6](-3.5,0.63)(-2.47,0.63)(-2.8,1.2)(-2.331,2.331)(-1.166,2.814)
\psline[linecolor=red,linewidth=0.6pt,arrows=->,arrowscale=1.6](-1.166,2.814)(0,3.297)(2.331,2.331)(2.566,1.766)

\rput{67.5}(-2.85,1.23){\pspolygon[linewidth=0.6pt,fillstyle=solid,fillcolor=black](0.2,0)(0,-0.1)(-0.2,0)}
\rput{-67.5}(2.85,1.23){\pspolygon[linewidth=0.6pt,fillstyle=solid,fillcolor=black](0.2,0)(0,-0.1)(-0.2,0)}

\rput{0}(0,3.25){\pspolygon[linewidth=0.6pt,fillstyle=solid,fillcolor=black](0.25,0)(0,0.1035)(-0.25,0)}
\rput(-2.3,2.3){\rput{45}(0,0){\pspolygon[linewidth=0.6pt,fillstyle=solid,fillcolor=black](0.25,0)(0,0.1035)(-0.25,0)}}
\rput(2.3,2.3){\rput{-45}(0,0){\pspolygon[linewidth=0.6pt,fillstyle=solid,fillcolor=black](0.25,0)(0,0.1035)(-0.25,0)}}

\rput(-3.1,1.4){\rput{22.5}(0,0){{1}}}
\rput(3.1,1.4){\rput{22.5}(0,0){{2}}}
     }
}}

\end{picture}
\caption{The true s-triangles in the figure (black triangles labeled by "1"{} and "2"{}) form a pair. The small p-ring associated with this pair is shown by a (partly dashed and partly solid) broken line. A part of this ring is traversed by the compressed flow generated by the triangles 1 and 2. Some v-triangles at the vertices of the ring are also shown. The main diagonal of the octagon orthogonal to the incident flow is shown dashed; the triangles in the pair are mutually symmetric with respect to this diagonal.}
\label{fig:pairing}
\end{figure}

Recall that the collections of vectors $\{ V_\iii \}$ indicate the directions of the thin flows for all s-sets (the vectors $(1,0)$ and $(-1,0)$ correspond to the directions "to the right"{} and "to the left"{}, respectively). The collections of intervals $\{ \UUU_\iii \}$ are responsible for the heights and widths of the thin flows. Also recall that each s-set is the disjoint union of two domains, the larger one and the smaller one.

Each s-set naturally induces two ({\it large} and {\it small}) p-rings; namely, the minimal p-rings containing the larger and the smaller domain, respectively. In the same way, the false triangles induce a p-ring, the same for all triangles. The large p-rings induced by two s-sets coincide, if these s-sets are responsible for incident flows with the same or complementary ($\al$ and $\pi - \al$) angles of incidence (and therefore are at the same level), and do not intersect otherwise. Also, large p-rings do not intersect small p-rings, and both do not intersect the p-ring induced by the false triangles.

The small p-rings induced by two s-sets coincide, if these s-sets are at the same level and the two intervals corresponding to the s-sets coincide. The small p-rings are always disjoint, if the corresponding intervals are disjoint.

Figure~\ref{fig:unilevel} may be viewed as a small part of a larger picture, where a part of a sequence of s-triangles at the same level above a certain polygonal contour is shown. The wider strip shown in light gray is a part of the large p-ring induced by these triangles, and the two thin strips are parts of the small p-rings induced by the s-sets labeled by "1"{} and "2".

Choose the collection of vectors $\{ V_\iii \}$ in such a way that the vectors corresponding to the s-triangles in a pair are opposite (that is, one of them is $(1,0)$ and the other $(-1,0)$). In Fig.~\ref{fig:pairing}, the vectors corresponding to the triangles 1 and 2 are $(-1,0)$ and $(1,0)$, respectively. Choose the collections of intervals $\{ \UUU_\iii \}$ so that the intervals corresponding to paired s-triangles coincide, and are disjoint otherwise. This implies, in particular, that the small p-rings induced by a pair of s-sets coincide.

For each small p-ring induced by a pair of s-sets consider the associated v-sets. (In Fig.~\ref{fig:pairing} the induced octagonal ring is indicated by a broken line, and 3 of 8 associated v-triangles are shown.) The minimal p-ring containing the small p-ring and these v-sets will be called the {\it extended small p-ring induced by the pair of s-sets}; note that its width is twice as much as that of the original small p-ring.

Require, additionally, that the extended small p-rings induced by different pairs do not intersect. It suffices to require that the distance between nearest intervals in the collection $\{ \UUU_\iii \}$ is larger than the sum of their lengths.

The above conditions imposed on $\{ V_\iii \}$ and $\{ \UUU_\iii \}$ guarantee that a portion of an incident flow, after four reflections from a hollow and an s-set, turns into a thin flow, then goes along the induced p-ring making reflections from the associated v-sets, then makes again four reflections from the paired s-set and the corresponding hollow, and finally is transformed into another portion of a flow, which is a continuation of the original one (see Fig.~\ref{fig:pairing}). Moreover, each trajectory in this portion is symmetric with respect to the diagonal of $P_1$ orthogonal to the incident flow, and therefore is invisible.

\begin{opr}\rm
The collection of sets including all the sets (the hollows, true and false s-triangles, and s-sets) in the $[\ve,r,\bbb_\iii,V_\iii,\UUU_\iii]$-systems, as well as all v-sets associated with the small p-rings, is called an $(\ve,r)$-{\it polygonal system resting on the segments} $\SSS_\iii$, or for brevity a {\it \psystem}. The bases of the (true and false) s-triangles, as well as the lateral sides of the v-triangles, are called the {\it segments generated by the p-system}. The corresponding trapezoids (resting on the segments $\SSS_\iii$) are called the {\it trapezoids of the p-system}. The union of these trapezoids is called the {\it envelope of this p-system}. The polygonal contour containing the segments $\SSS_\iii$ is called the {\it polygonal contour of the p-system}.
\end{opr}

Slightly abusing the language, we also say that the p-system is resting on the corresponding collection of hollows.

All segments generated by a p-system lie on polygonal contours, and their union is symmetric with respect to all lines of symmetry of $P_1$.

A fragment of a p-system (including a trapezoids resting on several hollows, two s-triangles and two v-triangles) is schematically represented in Fig.~\ref{fig:SSB}. The segments generated by the p-system are shown in bold.

\begin{figure}[h]
\begin{picture}(0,120)
\rput(8,1.9){
\scalebox{1}{

\psline[linewidth=0.6pt](-6,-1.732)(-5,0)(5,0)(6,-1.732)
\psline[linewidth=0.2pt,linestyle=dashed](6.758,-1.73)(5.379,0.656)(-5.379,0.656)(-6.758,-1.73)
\pspolygon[linewidth=0pt,linecolor=white,fillstyle=solid,fillcolor=white](-1.5,-0.05)(-1.5,0.05)(1.5,0.05)(1.5,-0.05)
\pspolygon[linewidth=0.4pt,linestyle=dashed](-1.5,0)(-4,1.5)(4,1.5)(1.5,0)
\psarc[linewidth=0.6pt](0,0.5){0.707}{-135}{-45}
\psarc[linewidth=0.6pt](-1,0.5){0.707}{-135}{-45}
\psarc[linewidth=0.6pt](1,0.5){0.707}{-135}{-45}
\pspolygon[linewidth=0.2pt,fillstyle=solid,fillcolor=lightgray](0.2,1)(1,0.5)(1.8,1)
\pspolygon[linewidth=0.2pt,fillstyle=solid,fillcolor=lightgray](-0.2,1)(-1,0.5)(-1.8,1)
\pspolygon[linewidth=0.2pt,fillstyle=solid,fillcolor=lightgray](5,0.75)(5.433,0.75)(5.65,0.375)
\pspolygon[linewidth=0.2pt,fillstyle=solid,fillcolor=lightgray](-5,0.75)(-5.433,0.75)(-5.65,0.375)
\psline[linewidth=1.2pt](0.2,1)(1.8,1)
\psline[linewidth=1.2pt](-0.2,1)(-1.8,1)
\psline[linewidth=1.2pt](5,0.75)(5.433,0.75)(5.65,0.375)
\psline[linewidth=1.2pt](-5,0.75)(-5.433,0.75)(-5.65,0.375)
\psline[linecolor=red,arrows=->,arrowscale=1.5,linewidth=0.8pt](-4.25,1.75)(-1.05,-0.18)(0.58,0.75)
\psline[linecolor=red,arrows=->,arrowscale=1.5,linewidth=0.8pt](1.25,0.656)(4,0.656)
\psline[linecolor=red,linewidth=0.2pt](0.58,0.75)(0.77,0.86)(1.2,0.86)(1,0.656)(1.25,0.656)
\psline[linecolor=red,linewidth=0.8pt](4,0.656)(5.15,0.656)
\psline[linecolor=red,arrows=->,arrowscale=1.5,linewidth=0.8pt](5.5,0.45)(6.758,-1.73)
}
}

\end{picture}
\caption{A fragment of a p-system including two s-triangles and two v-triangles. The trapezoid resting on three hollows is shown dashed, and an initial part of a billiard trajectory is also shown.}
\label{fig:SSB}
\end{figure}
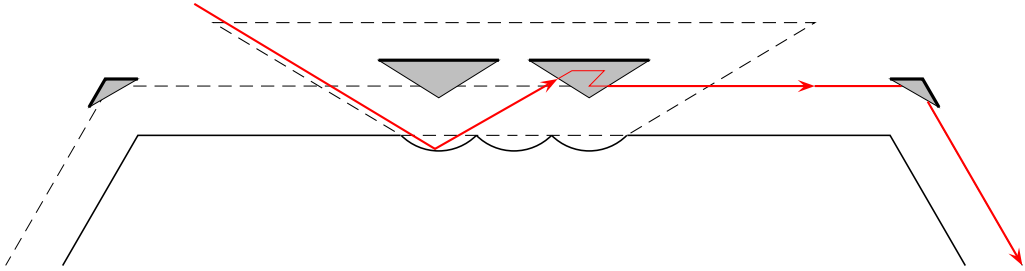

The minimal p-ring containing the trapezoids is called the {\it full p-ring of this p-system}. The large p-rings and the extended small p-rings induced by the s-sets, as well as the p-ring induced by the false s-triangles are called the {\it occupied p-rings of the p-system}. They are disjoint and lie in the full p-ring. The complement of the union of occupied p-rings in the full p-ring is again the union of finitely many rings, and is called the {\it free p-rings of the p-system}. Thus, the free and occupied p-rings are disjoint, and their union is the full ring of the system. (Moreover, the smallest and the largest p-rings in this union are free p-rings.) The free p-rings do not intersect the sets forming the p-system, 
and all segments generated by the p-system are contained in the inner boundaries of free p-rings.

Let a line segment lie on a polygonal contour. The isosceles triangle outside the contour with the angle $\pi/(2n)$ at the base and with the base coinciding with the segment is called the {\it semi-shadow of the segment} (see Fig.~\ref{fig:semishad}\,(a)). Taking if necessary $\ve$ sufficiently small, one can ensure that the semi-shadows of all segments generated by the \psystem\ lie in the free p-rings of this system.

Take again a polygonal contour and consider the $2n$-gon formed by the external bisectors of its vertices. The part of the $2n$-gon outside the contour is called the {\it the semi-shadow of the polygonal contour} (shown gray in Fig.~\ref{fig:semishad}\,(b)). (Thus, the semi-shadow of a polygonal contour is the union of the semi-shadows of its sides.) Now consecutively enumerate the vertices and consider the union of the two $n$-gons formed by the external bisectors through the even and through the odd vertices.
It is called the {\it weak semi-shadow of the contour} (the hexagonal star in Fig.~\ref{fig:semishad}\,(b)).

\begin{figure}
\begin{picture}(0,140)
\scalebox{1.2}{

\rput(0.8,2.2){
\scalebox{1.8}{
\rput(1.25,-0.25){\pspolygon[linewidth=0.3pt,fillstyle=solid,fillcolor=white](1,0)(0.5,0.866)(-0.5,0.866)(-1,0)(-0.5,-0.866)(0.5,-0.866)
\pspolygon[linewidth=0.1pt,fillstyle=solid,fillcolor=lightgray](-0.875,0.216)(-0.875,0.505)(-0.625,0.65)
\psline[linewidth=0.8pt](-0.875,0.216)(-0.625,0.65)
\pspolygon[linewidth=0.1pt,fillstyle=solid,fillcolor=lightgray](-0.875,-0.216)(-0.875,-0.505)(-0.625,-0.65)
\psline[linewidth=0.8pt](-0.875,-0.216)(-0.625,-0.65)
\rput(-0.06,0){
\pspolygon[linewidth=0.1pt,fillstyle=solid,fillcolor=lightgray](0.75,0.416)(1,0)(1,0.289)
\psline[linewidth=0.8pt](0.75,0.416)(1,0)
\pspolygon[linewidth=0.1pt,fillstyle=solid,fillcolor=lightgray](0.75,-0.416)(1,0)(1,-0.289)
\psline[linewidth=0.8pt](1,0)(0.75,-0.416)
}
}}
\rput(-0.6,-1.5){\scalebox{0.83}{(a)}}

\rput(2,-0.4){
\rput(2.5,0.1){
\scalebox{0.56}{
\rput{30}(7,0.1){
\scalebox{2.1}{
\pspolygon[linewidth=0.26pt,linestyle=dashed,fillstyle=solid,fillcolor=gray](1,0)(0.5,0.866)(-0.5,0.866)(-1,0)(-0.5,-0.866)(0.5,-0.866)
\pspolygon[linewidth=0.26pt,linestyle=dashed,fillstyle=solid,fillcolor=lightgray](0.5,0.866)(1.5,0.866)(1,0)
\pspolygon[linewidth=0.26pt,linestyle=dashed,fillstyle=solid,fillcolor=lightgray](-0.5,0.866)(-1.5,0.866)(-1,0)
\pspolygon[linewidth=0.26pt,linestyle=dashed,fillstyle=solid,fillcolor=lightgray](0.5,-0.866)(1.5,-0.866)(1,0)
\pspolygon[linewidth=0.26pt,linestyle=dashed,fillstyle=solid,fillcolor=lightgray](-0.5,-0.866)(-1.5,-0.866)(-1,0)
\pspolygon[linewidth=0.26pt,linestyle=dashed,fillstyle=solid,fillcolor=lightgray](-0.5,-0.866)(0.5,-0.866)(0,-1.732)
\pspolygon[linewidth=0.26pt,linestyle=dashed,fillstyle=solid,fillcolor=lightgray](-0.5,0.866)(0.5,0.866)(0,1.732)
}}
\rput(7,0.2){
\scalebox{1.8}{
\pspolygon[linewidth=0.3pt,fillstyle=solid,fillcolor=white](1,0)(0.5,0.866)(-0.5,0.866)(-1,0)(-0.5,-0.866)(0.5,-0.866)
}}
}
}
\rput(4.3,-1.2){\scalebox{0.83}{(b)}}
}

}}
\end{picture}
\caption{(a) Some segments on a polygonal contour and their semi-shadows (shown light gray). (b) The semi-shadow (gray) and the weak semi-shadow (the hexagonal star) of a polygonal contour.}
\label{fig:semishad}
\end{figure}
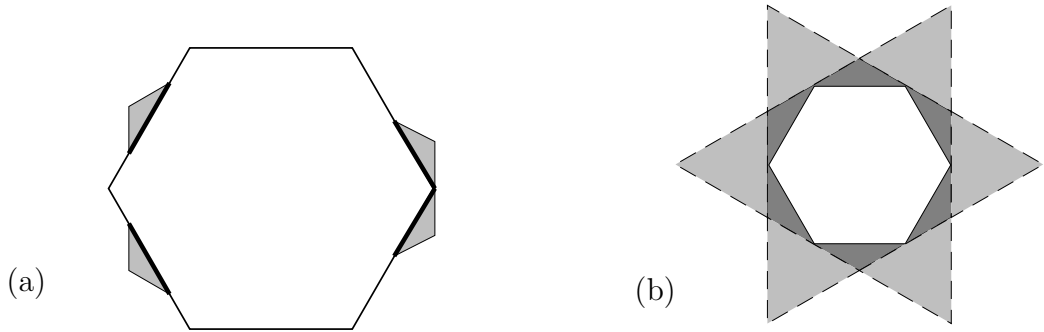

The semi-shadow of a segment generated by a p-system lies either in the envelope of the system (if it corresponds to an s-triangle), or in the complement of the semi-shadow of the corresponding polygonal contour (if it corresponds to a v-triangle).

The union of semi-shadows of the segments generated by a p-system is called the {\it semi-shadow generated by this p-system}. We conclude that it is the union of two subsets, with the former one lying in the envelope of the system, and the latter one in the complement of the semi-shadow of the polygonal contour.

Consider the billiard outside the union of all hollows, s-sets, v-sets, and false s-triangles of a p-system. For a particle of one of the $2n$ incident flows (except for finitely many particles) there are two possibilities.

(a) It makes the first reflection from the interior part of a hollow, then makes 3 reflections from an s-set, then goes along the corresponding polygonal ring making reflections from v-sets, then makes again 3 reflections from the paired s-set and a reflection from the corresponding hollow, and goes away along the same straight line as initially. The trajectory of this particle is invisible.

(b) It makes the first reflection from something else. We do not care about what happens after that.

In the case (a) the part of the trajectory between the first and fourth reflections (as well as its symmetric part, between the first and fourth reflections from the end) lies in the envelope of the p-system outside the semi-shadow generated by the system (see Fig.~\ref{fig:p-system}). The intermediate part of the trajectory, between the fourth reflection and the fourth reflection from the end, lies in an occupied p-ring of the p-system.

\begin{figure}
\begin{picture}(0,180)
\scalebox{0.8}{
\rput(13.5,2.6){

\psline[linewidth=0.7pt](-11,5)(2.887,5)(4.619,2)
\psline[linewidth=0pt,linecolor=white,fillstyle=solid,fillcolor=lightgray](-11,4)(2.309,4)(4.041,1)(3.753,0.5)(2.021,3.5)(-11,3.5)
\psline[linewidth=0.5pt](-11,4)(2.309,4)(4.041,1)
\psline[linewidth=0.5pt](-11,3.5)(2.021,3.5)(3.753,0.5)
\psline[linewidth=0pt,linecolor=white,fillstyle=solid,fillcolor=lightgray](-11,3.25)(1.876,3.25)(3.608,0.25)(3.464,0)(1.732,3)(-11,3)
\psline[linewidth=0.5pt](-11,3.25)(1.876,3.25)(3.608,0.25)
\psline[linewidth=0.5pt](-11,3)(1.732,3)(3.464,0)
\psline[linewidth=0pt,linecolor=white,fillstyle=solid,fillcolor=lightgray](-11,1.75)(1.01,1.75)(2.742,-1.25)(2.454,-1.75)(0.722,1.25)(-11,1.25)
\psline[linewidth=0.5pt](-11,1.75)(1.01,1.75)(2.742,-1.25)
\psline[linewidth=0.5pt](-11,1.25)(0.722,1.25)(2.454,-1.75)
\psline[linewidth=0pt,linecolor=white,fillstyle=solid,fillcolor=lightgray](-11,0.75)(0.433,0.75)(2.165,-2.25)(2.309,-2)(0.577,1)(-11,1)
\psline[linewidth=0.5pt](-11,1)(0.577,1)(2.309,-2)
\psline[linewidth=0.5pt](-11,0.75)(0.433,0.75)(2.165,-2.25)

\pspolygon[linewidth=0.5pt,linestyle=dashed,fillstyle=solid,fillcolor=lightgray](-7,1.75)(-6,3)(-5,1.75)
\pspolygon[linewidth=0.4pt](-6,0.5)(-7,1.75)(-5,1.75)
\pspolygon[linewidth=0.4pt](-6,4)(-5,2.75)(-4,4)

\psline[linewidth=0.7pt](-11,0)(-10,0)
\psline[linewidth=0.7pt](-6,0)(0,0)(1.732,-3)

\psarc[linewidth=0.7pt](-9,1.9){2.15}{-118}{-62}
\psarc[linewidth=0.7pt](-7,1.9){2.15}{-118}{-62}

\psline[linewidth=0.5pt,linestyle=dashed](-11,0.03)(-6,0.03)(-2,4.97)(-11,4.97)

\psline[linecolor=red,arrows=->,arrowscale=2,linewidth=0.9pt](-10.9,2)(-8.8,-0.25)(-7.4,1.5)
\psline[linecolor=red,linewidth=0.9pt](-7.4,1.5)(-5.65,3.7)(-4.5,3.7)(-5.1,3.125)(-4.7,3.125)
\psline[linecolor=red,arrows=->,arrowscale=2,linewidth=0.9pt](-4.7,3.125)(0,3.125)
\psline[linecolor=red,linewidth=0.9pt](0,3.125)(1.66,3.125)
\psline[linecolor=red,arrows=->,arrowscale=2,linewidth=0.9pt](1.804,3.125)(3.536,0.125)
\pspolygon[linewidth=0.7pt,fillstyle=solid,fillcolor=black](0.177,1)(0.577,1)(0.777,0.654)
\pspolygon[linewidth=0.7pt,fillstyle=solid,fillcolor=black](1.476,3.25)(1.876,3.25)(2.076,2.904)
\pspolygon[linewidth=0.3pt,fillstyle=solid,fillcolor=black](-6,4)(-4,4)(-4.4,3.5)(-4.5,3.75)(-4.55,3.85)(-5.4,3.85)(-5.6,3.75)(-5.73,3.67)
\pspolygon[linewidth=0.3pt,fillstyle=solid,fillcolor=black](-5,3)(-5.16,3.16)(-5.16,3)
\rput(-1.07,-2.25){
\pspolygon[linewidth=0.3pt,fillstyle=solid,fillcolor=black](-6,4)(-4,4)(-4.4,3.5)(-4.5,3.75)(-4.55,3.85)(-5.2,3.85)(-5.4,3.75)(-5.53,3.67)(-5.73,3.67)
\pspolygon[linewidth=0.3pt,fillstyle=solid,fillcolor=black](-4.95,3)(-5.11,3.16)(-5.11,3)
}
}}
\end{picture}
\caption{Fragments of the following objects are shown: a trapezoid of a p-system (a part of its boundary is shown by dashed line); two v-triangles of this system (shown black); two s-sets (shown black) and the corresponding s-triangles; the semi-shadow of an s-triangle (shown light-gray and bounded by dashed line); the full p-ring of the system composed of 5 free and 4 occupied (two large and two extended small) p-rings. A fragment of a billiard trajectory is also shown.}
\label{fig:p-system}
\end{figure}
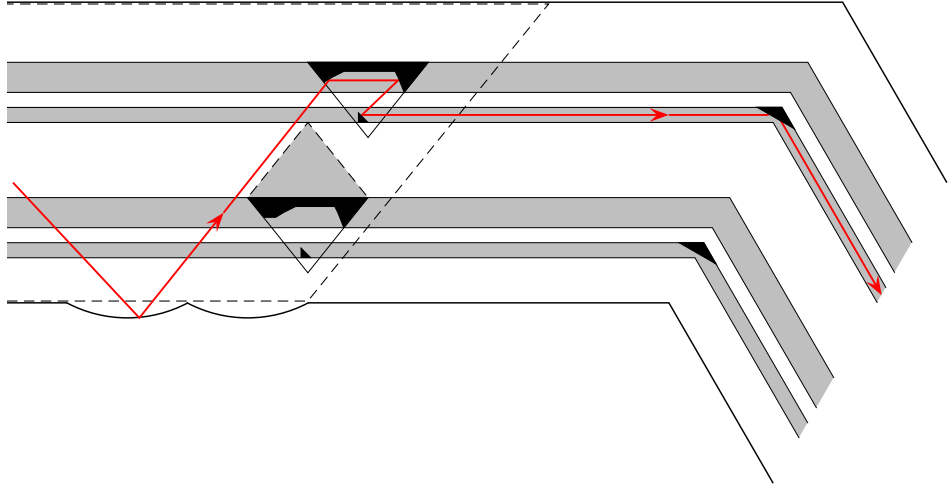

\section{Constructing a hierarchy of polygonal systems}

Take polygons $P$ and $P'$ homothetic to $P_1$ with the center at $O$ and such that $P'$ lies in the weak semi-shadow of $P$ and $P_1 \subset P \subset P' \subset P_2$ (and moreover, both $P \setminus P_1$ and $P' \setminus P$ have nonempty interior).

Now determine an iterative procedure leading to the construction of an invisible set. At each step $i = 0,\, 1,\, 2,\ldots$ of the procedure we inductively define a set $B_i$ and a marked part of its boundary $J_i$. The marked part of boundary is responsible for "visibility"{} of the set $B_i$, and its length $|J_i|$ goes to zero as $i \to \infty$. We also require that $J_i$ is the finite union of line segments lying on polygonal contours, and that these contours are inner boundaries of p-rings disjoint with $B_i$. Both $B_i$ and $J_i$ are symmetric with respect to all lines of symmetry of $P$, and additionally, $P_1 \subset B_i \subset P'$ and $\pl B_i \cap \pl P_1 = \emptyset = \pl B_i \cap \pl P'$.   

Initially we have $B_0 = P$ and $J_0 = \pl P$; that is, the original set coincides with the polygon and all its boundary is marked.

At the $i$th step of the procedure we do the following.
\vspace{2mm}

1) Take a sub-interval of each interval in $J_i$, and let $\tilde J_i$ be the union of these sub-intervals. We require that $\tilde J_i$ is symmetric with respect to all lines of symmetry of $P$ and the total length of the remaining part of $J_i$ is smaller than $\bt_i$,\, $|J_i \setminus \tilde J_i| < \bt_i$, where $\lim_{i\to\infty} \bt_i =0$.

2) By the inductive hypothesis, there are finitely many p-rings disjoint with $B_i$ with the inner boundaries containing all the sub-intervals. Without loss of generality we assume that all these p-rings are contained in $P'$. For each of these p-rings, take an $(\ve,r)$-polygonal system (p-system) resting on the corresponding sub-intervals, with $\ve$ and $r$ being the same for all p-systems. Take the parameter $r$ so small that each trapezoid of these p-systems lies in the corresponding p-ring, and also lies in the semi-shadow of the corresponding interval of $J_i$. Take $\ve$ so small that the total length of the segments generated by all the p-systems is smaller than $\bt_i$, and the hollows of the p-systems lie outside $P_1$.

3) Let $\tilde B_i$ be the modification of $B_i$ obtained by substituting the chosen sub-intervals by the added circular hollows. We have $\tilde B_i \subset B_i$. The new set $B_{i+1}$ is the union of the modified set $\tilde B_i$ and the sets (s-sets, v-sets, and false s-triangles) forming the added p-systems. The new marked part of boundary, $J_{i+1}$, is the union of $J_i \setminus \tilde J_i$ and the segments generated by these p-systems.
    \vspace{2mm}

A part of $J_{i+1}$ lies in the part of $\pl B_i$ which is not substituted with hollows, and therefore lies in $\pl B_{i+1}$. The remaining part of $J_{i+1}$ lies on the boundary of the added sets. It follows that $J_{i+1} \subset \pl B_{i+1}$.

Due to the construction of a p-system, the union of segments generated by the added p-systems is the finite union of disjoint segments and is symmetric with respect to all lines of symmetry of $P$. Hence $J_{i+1}$ also satisfies these conditions.

It follows from the construction that the full p-rings of the added p-systems are disjoint. Each segment in $J_i \setminus \tilde J_i$ lies in the inner boundary of the full p-ring of an added p-system, and therefore also lies in the inner boundary of a free ring of this system. On the other hand, each segment generated by a p-system also lies on the inner boundary of a free ring of the system. These free rings do not contain points of $B_i$, and also do not contain points of the added sets. Thus, all inductive assumptions for the step $i+1$ are satisfied.

That is, we start with a unique domain --- the polygon $P$. At each step of the procedure we reduce the existing domains (by making hollows on their boundary) and add new domains (the sets of the added p-systems) that are mutually disjoint and are also disjoint with the existing ones. The resulting set $B$ is the union of all reduced domains obtained in this procedure. In exact terms it can be defined as $B = \cup_{k=1}^\infty \cap_{i \ge k} B_i$, or (equivalently) as $B = \cap_{k=1}^\infty \cup_{i \ge k} B_i$.


The p-systems added in the course of the procedure form a tree (see Fig.~\ref{fig:tree}). There is a unique p-system at the first step; it generates several p-systems of the second step; each of them in turn gives rise to several p-systems of the third step, etc.

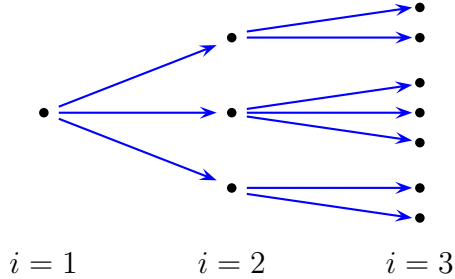
\begin{figure}
\begin{picture}(0,105)
\scalebox{1}{
\rput(5,2){

\psline[linecolor=blue,arrows=->,arrowscale=1.5,linewidth=0.8pt](0.2,0.08)(2.3,0.9)
\psline[linecolor=blue,arrows=->,arrowscale=1.5,linewidth=0.8pt](0.2,0)(2.3,0)
\psline[linecolor=blue,arrows=->,arrowscale=1.5,linewidth=0.8pt](0.2,-0.08)(2.3,-0.9)

\psline[linecolor=blue,arrows=->,arrowscale=1.5,linewidth=0.8pt](2.7,1.08)(4.9,1.4)
\psline[linecolor=blue,arrows=->,arrowscale=1.5,linewidth=0.8pt](2.7,1)(4.9,1)
\psline[linecolor=blue,arrows=->,arrowscale=1.5,linewidth=0.8pt](2.7,-1.08)(4.9,-1.4)
\psline[linecolor=blue,arrows=->,arrowscale=1.5,linewidth=0.8pt](2.7,-1)(4.9,-1)
\psline[linecolor=blue,arrows=->,arrowscale=1.5,linewidth=0.8pt](2.7,-0.05)(4.9,-0.37)
\psline[linecolor=blue,arrows=->,arrowscale=1.5,linewidth=0.8pt](2.7,0.05)(4.9,0.37)
\psline[linecolor=blue,arrows=->,arrowscale=1.5,linewidth=0.8pt](2.7,0)(4.9,0)

\psdots(0,0)
(2.5,1)(2.5,0)(2.5,-1)
(5,1.4)(5,1)(5,0.4)(5,0)(5,-0.4)(5,-1)(5,-1.4)

\rput(0,-2){\scalebox{1}{$i=1$}}
\rput(2.5,-2){\scalebox{1}{$i=2$}}
\rput(5,-2){\scalebox{1}{$i=3$}}
}}
\end{picture}
\caption{The hierarchy of p-systems of the procedure.}
\label{fig:tree}
\end{figure}

In a pair "parent -- child", the full p-ring of the p-system child is contained in a free p-ring of the p-system parent. It follows that in a pair "ancestor -- successor", the full p-ring of the p-system successor is contained in a free p-ring of the p-system ancestor. The polygonal contour of the p-system child lies outside the polygonal contour of the p-system parent. It follows that the polygonal contour of the the p-system successor lies outside the polygonal contour of the p-system ancestor. Further, the envelope of the p-system child is contained either in the envelope of the p-system parent (and moreover, in the semi-shadow generated by this system), or outside the semi-shadow of its polygonal contour. It follows that the same is true for a pair "ancestor -- successor", and in particular, the envelope of the p-system successor is contained either in the semi-shadow generated by the p-system ancestor, or outside the envelope of this system.

Further, if two p-systems do not form a pair "ancestor -- successor", then they are successors of a common p-system ancestor, and therefore their full p-rings lie in disjoint free p-rings of this system.

Now consider a p-system from the tree and a billiard trajectory outside the union of the sets (hollows, s- and v-sets, and false s-triangles) forming this system. Assume that the first reflection of this trajectory is from the interior part of a hollow (and therefore it satisfies condition (a) at the end of the previous section).

\begin{lemma}\label{l:traj}
The part of the trajectory between the first and the last reflections does not intersect the sets of the other p-systems from the tree.
\end{lemma}

\begin{proof}
We need to prove that this part of the trajectory does not intersect sets of (a) p-systems ancestors; (b) p-systems that are neither ancestors nor successors; (c) p-systems successors.

(a) This part of the trajectory lies in the full ring of the system, and therefore also lies in a free p-ring of each p-system ancestor; hence it does not intersect the sets of p-systems ancestors.

(b) Further, if a system does not form a pair "ancestor -- successor"{} with the given p-system, then the full p-rings of these systems are disjoint, and therefore the sets of that system do not intersect the given part of the trajectory.

(c) Finally, the part of the trajectory between the first and the fourth reflections, as well as the symmetric part between the first and the fourth reflections from the end, lie in the envelope of the system outside the semi-shadows generated by it, and therefore does not intersect the sets of each p-system successor. On the other hand, the part of the trajectory between the fourth reflection and the fourth reflection from the end lies in an occupied p-ring of the p-system, and therefore does not intersect the sets of each p-system successor.
\end{proof}

The following lemma finishes the proof of Theorem \ref{t2}.

\begin{lemma}\label{l:invis}
The set $B$ is invisible for the $2n$ incident flows.
\end{lemma}

\begin{proof}
First consider the s- and v-triangles included in the p-systems of the tree. We are going to show that the lateral sides of the s-triangles and the bases of the v-triangles are not accessible for the first reflection of the $2n$ incident flows.

Indeed, each v-triangle lies in the weak semi-shadow of $P$, and therefore is protected by $P$ from the first reflection of $n-1$ incident flows. Further, its base is protected from the flows with the opposite $n-1$ directions by the v-triangle itself. Finally, it is not accessible for the two remaining opposite directions parallel to the base. Thus, the base is shielded from the first reflection.

Further, the lateral sides of an s-triangle are shielded by the s-triangle itself from the first reflection of $n$ incident flows. They are also protected from the flows with the opposite $n$ directions, since the s-triangle is situated in the semi-shadow generated by the p-system parent. Thus, the lateral sides of the s-triangle are not accessible for the first reflection.

$B$ is the disjoint union of infinitely many domains of three kinds. The (unique) domain of the first kind is the reduced polygon $P$, with its boundary substituted by the union of infinitely many circular hollows. Domains of the second kind are reduced s-sets, with the base of each s-set substituted by the union of infinitely many hollows. Domains of the third kind are reduced v-sets, where the lateral sides of each v-set are substituted by the union of infinitely many hollows. The argument in the beginning of the proof shows that only hollows are accessible for the first reflection.

Consider a particle of a flow in $B^c$. If it does not make reflections from $B$, there is nothing to do. If it does, then the first reflection is from a hollow. Take the p-system resting on this hollow, and consider the (auxiliary) trajectory of the particle in the complement of this p-system with the same initial data (and therefore also with the same point of the first reflection) as the original particle. Lemma \ref{l:traj} guarantees that the part of the trajectory between the first and the last reflections does not intersect sets of the other p-systems, and therefore of course does not intersect the corresponding reduced sets. Since the part of its trajectory before the first reflection (a half-line) coincides with the original trajectory, it does not intersect the reduced sets of other p-systems, and the same is true for the symmetric part of the trajectory (after the last reflection).

It follows that the auxiliary trajectory coincides with the original one (in $B^c$), and since the former trajectory is invisible, so is the latter one. Hence $B$ is invisible for the $2n$ incident flows.

Thus, Lemma \ref{l:invis}, and therefore also Theorem \ref{t2}, are proved.
\end{proof}

\section*{Acknowledgements}

This work was supported by Portuguese funds through CIDMA -- Center for Research and Development in Mathematics and Applications and FCT -- Portuguese Foundation for Science and Technology, within the project PEst-OE/MAT/UI4106/2014, as well as by the FCT research project PTDC/MAT/113470/2009.


\end{document}